%% file: main.tex
\documentclass[11pt]{article}
\usepackage[T1]{fontenc}
\usepackage[utf8]{inputenc}
\usepackage{lmodern}
\usepackage[margin=1in]{geometry}
\usepackage{amsmath,amssymb,amsthm,mathtools,mathrsfs}
\usepackage{booktabs,array,longtable}
\usepackage{enumitem}
\usepackage{microtype}
\usepackage[numbers,sort&compress]{natbib}
\usepackage[dvipsnames]{xcolor}
\usepackage{tikz}
\usepackage{xurl}
\usepackage[colorlinks=true,linkcolor=MidnightBlue,citecolor=MidnightBlue,urlcolor=MidnightBlue]{hyperref}
\usepackage[nameinlink,noabbrev]{cleveref}
\hypersetup{pdftitle={Counting the Fixed Points of Binary Belief Propagation: Generic Degree, Syzygies, and Real Solutions},pdfauthor={Yusuke Watanabe}}
\setlist[enumerate]{itemsep=3pt,topsep=4pt}
\setlist[itemize]{itemsep=3pt,topsep=4pt}
\numberwithin{equation}{section}
\newtheorem{theorem}{Theorem}[section]
\newtheorem{proposition}[theorem]{Proposition}
\newtheorem{lemma}[theorem]{Lemma}
\newtheorem{corollary}[theorem]{Corollary}
\theoremstyle{definition}
\newtheorem{definition}[theorem]{Definition}
\newtheorem{example}[theorem]{Example}
\theoremstyle{remark}
\newtheorem{remark}[theorem]{Remark}
\input{sections/macros}
\title{Counting the Fixed Points of Binary Belief Propagation:\\
 Generic Degree, Syzygies, and Real Solutions}
\author{Yusuke Watanabe}
\date{27 October 2026}
\begin{document}
\maketitle
\begin{abstract}
\input{sections/abstract}
\end{abstract}
\medskip
\noindent\textbf{Keywords.} \input{sections/keywords}

\noindent\textbf{MSC2020.} \input{sections/msc}
\input{sections/introduction}
\input{sections/bp-degree}
\input{sections/incidence-algebra}
\input{sections/strong-field}
\input{sections/real-sectors}
\input{sections/discussion}
\input{sections/declarations}

\bibliographystyle{plainnat}
\bibliography{paper-references}
\clearpage
\appendix
\renewcommand{\thesection}{SM\arabic{section}}
\section*{Supplementary Materials}
\addcontentsline{toc}{section}{Supplementary Materials}
\input{supplement/supplement}
\end{document}

%% file: sections/macros.tex
\newcommand{\C}{\mathbb C}
\newcommand{\R}{\mathbb R}
\newcommand{\F}{\mathbb F}
\newcommand{\Pp}{\mathbb P}
\newcommand{\A}{\mathbb A}
\newcommand{\U}{\mathscr U}
\newcommand{\cO}{\mathcal O}
\newcommand{\cK}{\mathcal K}

\DeclareMathOperator{\Tor}{Tor}
\DeclareMathOperator{\Hilb}{Hilb}

\DeclareMathOperator{\sgn}{sgn}

\DeclareMathOperator{\depth}{depth}

\DeclareMathOperator{\Sym}{Sym}
\DeclareMathOperator{\Frac}{Frac}

\DeclareMathOperator{\diag}{diag}

\DeclareMathOperator{\wt}{wt}
\newcommand{\romega}{\widetilde{\omega}}

%% file: sections/abstract.tex
How many fixed points does belief propagation have on a binary pairwise
model? Over complex parameters, the admissible equations of a finite
loopless multigraph form an algebraic family whose generic degree we
determine: it is a graph state sum over edge sets without vertices of
degree one, each vertex of positive degree contributing one less than
its degree. We give
two geometric realizations of this state sum. The first is the
total Betti number of a universal edge-incidence ring constructed
without parameters. Its full graded Betti table is the state sum,
its Tor algebra is natural for noncontracting graph maps, the
alternating numerator recovers the graph polynomial of Watanabe and
Fukumizu, and an Artin reduction by a regular sequence realizes the one-variable
specialization, corrected for tree components. The second lies in the
actual family: at the strong-field boundary each state factor is a
ramification index, and at sufficiently strong field the state charts
exhaust the admissible complex fiber. For fixed nonzero real couplings, every
Eulerian message-sign sector contains a real fixed point at every
external field, and strong fields attain this lower bound: one
admissible real fixed point per element of the binary cycle space and
one positive fixed point. For graphs of maximum degree at most three,
this minimum equals the number of minimal module generators of the
edge-incidence ring. One state decomposition connects the complex
degree, graded syzygies, and an attained minimum real count. The degree,
selected ring and resolution statements, and the real counts are
formalized in Lean 4.

%% file: sections/keywords.tex
belief propagation, Bethe approximation, generic degree, graph polynomial, Koszul complex, minimal free resolution, Cohen--Macaulay ring, real solutions, binary cycle space, Ising model

%% file: sections/msc.tex
13D02, 05C31, 13C14, 14P05, 62H22, 68T37, 82B20

%% file: sections/introduction.tex
\section{Introduction}
\label{sec:introduction}

How many fixed points can belief propagation (BP) have? For a binary
pairwise model on a finite loopless multigraph, BP associates rational
fixed-point equations with the graph and its local interaction
parameters. Positive solutions describe normalized local beliefs and
stationary points of the Bethe free energy. Over complex parameters the
same equations define an algebraic family, and our first theorem
determines its generic degree: for generic parameters, the number
\(N(G)\) of admissible complex fixed points of the specified
factorization equals a graph state sum \(Q(G)\) over edge sets without
degree-one vertices.

The state sum then receives two geometric realizations. The first is
the total Betti number of a commutative ring constructed from edge
incidences alone, with no BP parameters: every state contribution
appears as a graded Betti number of its minimal free resolution, and
the ring carries products and graph maps that a mere count does not:
classes of vertex-disjoint cycles multiply, while cycles sharing a
vertex have product zero. The second lies in the actual BP family: at the strong-field
boundary each state factor is a ramification index, and at sufficiently
strong field the state charts exhaust the admissible complex fiber. Our second theorem
uses this boundary model, together with a global description of the
real sign sectors, to determine the minimum number of real fixed points
for every fixed nonzero real coupling vector. The common state
decomposition also gives a direct bridge: on graphs of maximum degree
at most three, the minimum real count equals the number of minimal
module generators of the ring.

Why count fixed points? A fixed point is what the BP iteration returns
when it converges, and on a graph with cycles it need not be unique:
distinct positive fixed points are distinct interior stationary points
of the Bethe free energy \cite{yedidia2005bethe}, and the one reached
depends on the initialization. Uniqueness statements must therefore
control all fixed points at once, as in the signed-graph classification
of \cite{watanabe2011signed}, and enumeration by homotopy continuation
is organized around a bound on the number of complex solutions
\cite{knoll2018bp}. The generic degree \(N(G)\) is the invariant behind
both tasks: it bounds the numbers of real and positive fixed points for
generic parameters, and the real-sector analysis identifies the part of
this count forced to be real at every external field, one fixed point
per element of the binary cycle space. The complete graph \(K_4\), with
generic complex degree \(48\) and minimum real count \(8\), is the
worked example throughout.

\subsection{Two main results}

Let \(G=(V,E)\) be a finite loopless multigraph, with distinct parallel
edges retained as distinct interaction factors. Write \(n=|V|\),
\(m=|E|\), \(c(G)\) for the number of connected components, and
\(\nu(G)=m-n+c(G)\). Let \(\tau(G)\) count the tree components,
including isolated vertices. For an edge subset \(S\), let \(d_S(v)\) be
its degree at \(v\), and let \(v(S)\) count its incident vertices.
Set
\begin{equation}\label{intro:state-sum}
 \varphi(0)=1,\qquad \varphi(d)=d-1\quad(d\geq1),\qquad
 \wt(S)=\prod_{v\in V}\varphi(d_S(v)),\qquad
 Q(G)=\sum_{S\subseteq E}\wt(S).
\end{equation}
Only the empty state and edge sets without degree-one vertices
contribute. Let \(N(G)\) denote the generic degree of the BP equations on
the admissible open set of \cref{eq:admissibility}; in particular, zeros
introduced by clearing denominators are excluded.

To state the algebraic result, put \(X_G=(\Pp^1)^V\). For each edge
\(e=\{u,v\}\), take the line bundle of bidegree \((1,1)\) on its
endpoint factors, and let \(F_G\) be their direct sum, with the orientation
lines specified in \cref{alg:edge-bundles}. Let
\(S_G=\Sym H^0(X_G,F_G)^*\), a polynomial ring in \(4m\) variables.
The universal common-zero scheme \(Z_G\subset X_G\times H^0(X_G,F_G)\)
defines the \emph{universal edge-incidence ring}
\(\U_G=\Gamma(Z_G,\cO_{Z_G})\). Write
\[
 E_1^{-q,j}(G)=H^j(X_G,\bigwedge^q F_G^\vee).
\]
Finally, \(\theta_G(t,\gamma)\) and \(\omega_G(t)\) denote the graph
polynomials with the normalization in \cref{alg:omega-definition}.

\begin{theorem}[The edge-incidence ring and the BP degree]
\label{thm:structure}
For every finite loopless multigraph, \(N(G)=Q(G)\), and \(\U_G\) is a
canonical normal Cohen--Macaulay domain, finite over \(S_G\), of
dimension \(d_G=3m+n-\tau(G)\).
Its minimal graded free resolution has Betti numbers
\begin{equation}\label{intro:betti}
 \beta_{p,q}^{S_G}(\U_G)
 =\sum_{\substack{S\subseteq E\\ |S|=q,\ |S|-v(S)=p}}\wt(S),
 \qquad
 N(G)=\sum_{p,q}\beta_{p,q}^{S_G}(\U_G).
\end{equation}
There is a canonical isomorphism of bigraded algebras
\begin{equation}\label{intro:tor}
 \Tor_p^{S_G}(\U_G,\C)_q\cong E_1^{-q,q-p}(G).
\end{equation}
The construction and this identification are natural for endpoint-preserving
graph maps that do not contract edges. Moreover,
\begin{equation}\label{intro:hilbert}
 \Hilb_{\U_G}(t)
 =\frac{\theta_G(t,2\mathrm i)}{(1-t)^{4m}}
 =\frac{\omega_G(t)}{(1-t)^{3m+n}}.
\end{equation}
A quotient \(A_G\) by \(d_G\) general coefficient linear forms, which
form a regular sequence, is Artinian with
\[
 \Hilb_{A_G}(t)=\romega_G(t):=\frac{\omega_G(t)}{(1-t)^{\tau(G)}}
 =\prod_{C}\omega_C(t),
\]
the product running over the components \(C\) that contain a cycle.
In particular \(\romega_G=\omega_G\) when there are no tree components.
\end{theorem}

Thus one graded table carries three invariants: its unsigned total is
the BP degree, its alternating numerator is \(\theta_G(t,2\mathrm i)\),
and its Hilbert numerator after the regular cut is \(\romega_G(t)\).
The ring \(\U_G\) involves no choice of pencil or section; its Artin
reductions are not canonical and are not identified with BP root
algebras. For a tree, \(N(G)=1\), \(\U_G\cong S_G\), and
\(\romega_G=1\).

The second result concerns the actual BP parameters
\(a_e=e^{2J_e}\), \(b_i=e^{2h_i}\). For real \(J,h\), let
\(R_G(J,h)\) count all distinct admissible real roots, including roots
with negative messages and allowing the value \(\infty\). Let
\(N_+(J,h)\) count roots with all messages positive.

\begin{theorem}[Strong-field ramification and the minimum real count]
\label{thm:real}
Let \(G\) be a finite loopless multigraph.
\begin{enumerate}[label=\textup{(\roman*)}]
\item Fix complex edge parameters \(a_e\ne0\) with \(a_e^2\ne1\).
For each contributing state \(S\), the fixed-\(a\) admissible BP family
has a smooth local analytic extension across \(p=0\), where
\(p_i=b_i^{-1}\). In coordinates \(w=(w_i)\) on this extension,
the parameter map is exactly
\begin{equation}\label{intro:power-map}
 p_i=(-1)^{d_S(i)}w_i^{r_i},\qquad
 r_i=\begin{cases}1,&d_S(i)=0,\\d_S(i)-1,&d_S(i)\geq2.\end{cases}
\end{equation}
Its boundary local algebra is
\(\bigotimes_{i\in V}\C[w_i]/(w_i^{r_i})\), of length \(\wt(S)\).
For all sufficiently small nonzero \(p_i\), these charts exhaust the
complex admissible fiber, giving exactly \(Q(G)\) simple roots.

\item Fix arbitrary finite real couplings \(J_e\ne0\). At every
field \(h\in\R^V\), each Eulerian edge subset is the negative-message
support of at least one admissible real root, so
\(R_G(J,h)\geq2^{\nu(G)}\).
In every fixed field orthant, sufficiently large field magnitudes attain
this bound: the numbers of complex, real, and positive roots are
respectively
\begin{equation}\label{intro:three-counts}
 \bigl(N_{\C},R_G,N_+\bigr)=\bigl(Q(G),2^{\nu(G)},1\bigr),
\end{equation}
and all complex roots are simple. Consequently
\begin{equation}\label{intro:minimum}
 \min_{h\in\R^V}R_G(J,h)=2^{\nu(G)}.
\end{equation}
\end{enumerate}
\end{theorem}

Thus the same integer \(d-1\) is the dimension of
\(H^1(\Pp^1,\cO(-d))\) in the resolution and a ramification index of
the BP family at strong field; the comparison is statewise. The real
statement has no weak-coupling assumption: each sign sector is described
globally by vertex coordinates as a bounded perturbation of an invertible
diagonal map. The state charts do not use the ring; their exhaustion of
the complex fiber uses \(N=Q\).

The two theorems give the following numerical connection
(\cref{alg:generator-real-bridge}): if the maximum degree of \(G\) is
at most three, then for every fixed finite \(J\) with \(J_e\ne0\),
\begin{equation}\label{intro:generator-real-bridge}
 \sum_q\beta_{0,q}^{S_G}(\U_G)
 =2^{\nu(G)}=\min_h R_G(J,h).
\end{equation}
The left side counts vertex-disjoint cycle states, including the empty
state. In this degree range these are precisely the Eulerian states
which index the real sectors.

\subsection{A running example}

For \(K_4\) the resolution has ranks \(8,24,16\) in homological degrees
\(0,1,2\), so the generic complex BP degree is \(48\) and there are
\(8=1+4+3\) module generators, indexed by the empty state, the four
triangles, and the three four-cycles; the Artin reduction has Hilbert
series \(1+2t+3t^2+8t^3+16t^4\) and length \(30\), a different
invariant. At fixed nonzero real couplings the minimum real count is
\(2^3=8\), one root per Eulerian state, with one positive root at
strong field (\cref{fig:k4-states}, \cref{sf:k4-states}).

\input{sections/related-work}

\subsection{Organization}

\Cref{sec:degree} proves \(N=Q\); \cref{sec:incidence} constructs the
ring and its resolution; \cref{sec:strong-field} proves the strong-field
normal form and the exhaustion of the complex fiber; \cref{sec:real}
treats the real sectors. The supplementary material collects the
notation (\cref{supp:notation}), the elementary proofs omitted below,
and a longer comparison with related work.

\subsection{Formalization}

All clauses of \cref{thm:structure,thm:real} have been formalized in
Lean~4 with mathlib, under formal hypotheses recorded with each
statement: the generic degree in its counting form, and the geometric
objects through explicit models, namely the chart model of \(\U_G\),
a finite biproduct model of \(\bigwedge^qF_G^\vee\), and the
\v{C}ech model of \eqref{alg:zero-base-change}, including its products.
Several clauses are proved by arguments different from ours, notably
\cref{sf:degree-saturation}, which is replaced by a topological counting
argument, and the Gr\"obner-basis computation for \(K_4\) is external.
The accompanying repository~\cite{watanabe2026bplean} contains the
source code, build instructions, the precise formal hypotheses and, for
every statement of this paper, the corresponding formal declarations.

%% file: sections/related-work.tex
\subsection{Related work}
\label{intro:related}

Yedidia, Freeman, and Weiss identify BP fixed points with interior
stationary points of the Bethe free energy \cite{yedidia2005bethe}, and
Watanabe classifies the signed graphs with a unique positive fixed point
\cite{watanabe2011signed}; both concern positive messages. Knoll, Mehta,
Chen, and Pernkopf enumerate fixed points by homotopy continuation under
a different genericity \cite{knoll2018bp}. The index sets are the
generalized loops of loop calculus \cite{chertkov2006loopcalculusv2},
\(\theta_G\) and \(\omega_G\) are the polynomials of Watanabe and
Fukumizu \cite{watanabe2011graphpolynomials}, and the ring construction
follows the geometric method for syzygies \cite{weyman2003syzygies}.
Mixed multilinear systems \cite{bender2021mixedmultilinearv1} and
homological formulations of message passing
\cite{peltre2019homologicalv1,sergeantperthuis2025retractionsv1} are
compared in \cref{supp:related}.

%% file: sections/bp-degree.tex
\section{The admissible BP family and its generic degree}
\label{sec:degree}

\subsection{Coordinates and the open set being counted}

Throughout, a graph is finite, undirected, and loopless; parallel edges
are distinct factors. In particular, merging parallel interactions into a
single Ising coupling changes the BP equations and is not one of our
conventions. Let \(\vec E\) be the set of directed edges, and write
\(\bar\alpha\), \(t(\alpha)\), \(h(\alpha)\), and \(e(\alpha)\) for reversal,
tail, head, and the underlying undirected edge. Set
\[
 a_e=e^{2J_e},\qquad b_i=e^{2h_i}
\]
in the physical real model, and allow \(a_e,b_i\in\C^*\) independently
in its algebraic family. A directed message is represented by its ratio
\(y_\alpha\). Define the cavity ratio and the full vertex ratio by
\begin{equation}\label{eq:bp-ratios}
 X_\alpha=b_{t(\alpha)}
       \prod_{\substack{\gamma:\,h(\gamma)=t(\alpha)\\
                        \gamma\ne\bar\alpha}}y_\gamma,
 \qquad
 F_i=b_i\prod_{\gamma:\,h(\gamma)=i}y_\gamma.
\end{equation}
Only the reverse of the same edge is omitted from the cavity product.
For an incidence \((i,e)\), we also write
\(X_{i,e}=X_{i\to j,e}\), where \(j\) is the other endpoint of \(e\).
The rational BP equation is
\begin{equation}\label{eq:bp-update}
 y_\alpha=\frac{a_{e(\alpha)}X_\alpha+1}
                  {X_\alpha+a_{e(\alpha)}}.
\end{equation}
We use its polynomial residual
\[
 f_\alpha=y_\alpha(X_\alpha+a_{e(\alpha)})
                  -(a_{e(\alpha)}X_\alpha+1)
\]
on the open set
\begin{equation}\label{eq:admissibility}
 D=
 \prod_{\alpha\in\vec E}y_\alpha
 \prod_{\alpha\in\vec E}(X_\alpha+a_{e(\alpha)})
 \prod_{i\in V}(1+F_i)\ne0.
\end{equation}
The last factors ensure that vertex beliefs, and hence the locally
consistent edge beliefs below, can be normalized. All statements about
complex or real roots in this paper refer to this open set. Equivalently,
we use the localized coordinate ring, not the whole scheme obtained by
clearing denominators. A \emph{positive} root has every \(y_\alpha>0\)
and physical parameters \(a_e,b_i>0\).

\begin{example}[Why the last factor matters]
\label{ex:triangle-admissible}
Let \(G\) be a triangle with \(a_e=b_i=2\). A message travelling around
the cycle in one direction is updated by \(x\mapsto T(x)=(4x+1)/(2x+2)\),
and \(T^3(x)-x=-5(2x^2-2x-1)/(10x+4)\), so each orientation carries a
constant message \(x_+\) or \(x_-\), where \(x_\pm=(1\pm\sqrt3)/2\).
The rational BP equations thus have four solutions, all with nonzero
messages and update denominators. At every vertex
\(F_i=2x_{\mathrm{cw}}x_{\mathrm{ccw}}\), which is \(2\pm\sqrt3\) when
both orientations carry the same root and \(-1\) for the two mixed
solutions. The mixed solutions cannot be normalized to vertex beliefs
and are excluded by \(1+F_i\ne0\); the admissible count is
\(2=N(C_3)\). Since \(x_\pm\ne-1\), all four solutions admit the
sum-one message form \((y/(1+y),1/(1+y))\), so normalizability of
messages and of beliefs are different conditions. In general the
excluded locus \(F_i\equiv-1\) is nonempty of dimension \(\nu(G)-1\)
for generic parameters on a connected graph with a cycle
(\cref{supp:nonnormalizable}); the admissibility condition is what
makes the count finite (see also \cref{supp:related}).
\end{example}

Let \(P_G=(\C^*)^{m+n}\), and let
\[
 \pi_G:\mathcal X_G\longrightarrow P_G,\qquad
 \mathcal X_G=\{f_\alpha=0\text{ for all }\alpha,\ D\ne0\}
\]
be the scheme defined by these equations and localization.

\begin{proposition}\label{prop:bp-family}
The scheme \(\mathcal X_G\) is smooth and irreducible of dimension \(m+n\).
The map \(\pi_G\) is dominant and generically finite. There is a nonempty
Zariski open subset of \(P_G\) over which it is finite \'etale of a
well-defined positive degree \(N(G)\).
\end{proposition}

\begin{proof}
Write a normalized vertex belief as \((\xi_i,1-\xi_i)\) and the four
normalized entries of an edge belief as
\((q_e,\rho_e,\sigma_e,\delta_e)=(q_e,\xi_i-q_e,\xi_j-q_e,1-\xi_i-\xi_j+q_e)\).
On the open subset of \(\A^{n+m}\) where \(\xi_i(1-\xi_i)\ne0\) and
\(q_e\rho_e\sigma_e\delta_e\ne0\), adjoin variables \(a_e\) with
\begin{equation}\label{eq:belief-cover}
 a_e^2=\frac{q_e\delta_e}{\rho_e\sigma_e};
\end{equation}
this is a finite \'etale cover of rank \(2^m\). For a BP root put
\(X=X_{i,e}\), \(Y=X_{j,e}\), \(a=a_e\), and
\(Z_e=aXY+X+Y+a=(1+F_i)(Y+a)=(1+F_j)(X+a)\). Then
\(\xi_i=F_i/(1+F_i)\) and
\((q_e,\rho_e,\sigma_e,\delta_e)=Z_e^{-1}(aXY,X,Y,a)\) give a point of
the cover. Conversely, \(F_i=\xi_i/(1-\xi_i)\),
\(X_{i,e}=a_e\rho_e/\delta_e\), \(y_{j\to i,e}=F_i\delta_e/(a_e\rho_e)\),
and \(b_i=F_i/\prod_{e\ni i}y_{\mathrm{other}\to i,e}\) recover a BP
root: for instance \(X_{i,e}+a_e=a_e(1-\xi_j)/\delta_e\ne0\), and
\eqref{eq:belief-cover} verifies the update equations. All divisions are
by units, so these are inverse scheme maps. The cover is irreducible:
in \(\C(\xi,q)\) the valuation along \(q_e=0\) of the \(e\)th class
\(q_e\delta_e/(\rho_e\sigma_e)\) is one and that of every other class is
zero, so the \(m\) square classes are independent and their square
roots generate a field extension of degree \(2^m\); the covering ring is
free over its base and embeds in this field, hence is a domain. Thus
\(\mathcal X_G\) is smooth and irreducible of dimension \(m+n\).

At \(a_e=b_i=y_\alpha=1\) the Jacobian \(\partial f/\partial y\) equals
\(2I_{2m}\), so \(\pi_G\) is locally invertible there and dominant, and
equal dimensions give generic finiteness. Inverting the leading
coefficients of algebraic equations for a finite set of coordinate
generators makes \(\pi_G\) finite over a nonempty open subset of
\(P_G\); inverting a discriminant of the separable function-field
extension and imposing generic freeness makes it \'etale there, of
rank \(N(G)\).
\end{proof}

\subsection{Auxiliary edge coordinates}

For an incidence \((i,e)\), define
\begin{equation}\label{eq:aux-coordinates}
 A_{e,i}=\frac{X_{i,e}}{a_e},\qquad
 s_e=a_e^2,\qquad
 \zeta_i=\frac{b_i}{\prod_{e\ni i}a_e}.
\end{equation}
Here \(\zeta_i\) is the rescaled vertex parameter.
If \(d_i\) is the degree of \(i\), the BP equations imply
\begin{align}
 F_i&=A_{e,i}\frac{1+s_eA_{e,j}}{1+A_{e,j}},
 &F_j&=A_{e,j}\frac{1+s_eA_{e,i}}{1+A_{e,i}},
 \label{eq:aux-edge}\\
 \zeta_i&=F_i^{\,1-d_i}\prod_{e\ni i}A_{e,i}.
 \label{eq:aux-vertex}
\end{align}
For an isolated vertex the latter equation reads \(F_i=\zeta_i\).
The relevant open conditions are \(F_i\ne0,-1\) and
\(A_{e,i}\ne0,-1\).

These equations recover the original BP scheme on the open set. Indeed,
put
\begin{equation}\label{eq:aux-recovery}
 y_{j\to i,e}=\frac{F_i}{a_eA_{e,i}}.
\end{equation}
Equation \eqref{eq:aux-vertex} gives
\(b_i\prod_{e\ni i}y_{j\to i,e}=F_i\), and
\eqref{eq:aux-edge} gives the BP update with cavity ratio \(a_eA_{e,i}\).
Its denominator is \(a_e(1+A_{e,i})\ne0\). Conversely the original
BP equations give \eqref{eq:aux-edge}--\eqref{eq:aux-vertex}.
This local coordinate-ring isomorphism preserves multiplicities.
Choosing square roots \(a_e\) of fixed \(s_e\) does not multiply a
fiber's root count: each such choice determines a different
parameter point \((a,b)\).

Eliminating \(A_{e,j}\) from an edge gives a useful necessary equation:
\begin{equation}\label{eq:edge-quadratic}
 s_e A_{e,i}^2+
 \left[
  1-F_i+(s_e-1)\left(1-\frac{1+F_i}{1+F_j}\right)
 \right]A_{e,i}-F_i=0.
\end{equation}
The division by \(1+F_j\) is allowed. Also \(1+s_eA_{e,i}\ne0\),
since the other edge equation gives
\(F_j(1+A_{e,i})=A_{e,j}(1+s_eA_{e,i})\).
We use \eqref{eq:edge-quadratic} to control limits; the lifting argument
will always use both original edge equations.

\subsection{Excluding all weak-coupling boundary limits}

Let \(K\) be an algebraically closed field of characteristic zero and
\[
 \mathcal K=\overline{K((t))}.
\]
Extend the \(t\)-adic valuation to \(\mathcal K\), with \(v(t)=1\),
and write \(\mathcal O\) for its valuation ring. Such an extension
exists by \citet[Tag~00IA]{stacks2026cohomology}. Its residue field is
\(K\): the residue extension in each finite subextension of
\(\mathcal K/K((t))\) is finite, and \(K\) is algebraically closed
\citep[Tag~0ASH]{stacks2026cohomology}.
A bar denotes reduction \(\mathcal O\to K\).

\begin{lemma}[Boundary exclusion]\label{lem:weak-boundary}
Suppose \(v(s_e)=0\), \(\overline{s_e}=1\), and
\(v(\zeta_i)=0\), \(\overline{\zeta_i}\ne-1\).
Every admissible \(\mathcal K\)-solution satisfies
\[
 v(F_i)=v(A_{e,i})=0,\qquad \overline{F_i}\ne-1.
\]
The parameters \(\zeta_i\) may depend on \(t\).
\end{lemma}

\begin{proof}
We first remove poles of the full ratios by spin changes.
For \(\epsilon_i\in\{1,-1\}\), exchanging the two spin labels at
vertices with \(\epsilon_i=-1\) transforms
\[
 a'_e=a_e^{\epsilon_i\epsilon_j},\qquad
 b'_i=b_i^{\epsilon_i},\qquad
 y'_{j\to i,e}=y_{j\to i,e}^{\epsilon_i}.
\]
In the auxiliary coordinates this is
\begin{align}
 F'_i&=F_i^{\epsilon_i},&
 s'_e&=s_e^{\epsilon_i\epsilon_j},\nonumber\\
 A'_{e,i}&=
 s_e^{\epsilon_i(1-\epsilon_j)/2}A_{e,i}^{\epsilon_i},&
 \zeta'_i&=\zeta_i^{\epsilon_i}
       \prod_{\substack{e=\{i,j\}\\\epsilon_j=-1}}s_e^{\epsilon_i}.
 \label{eq:spin-gauge}
\end{align}
These transformations preserve admissibility and the equations.
All transformed \(s_e\) still reduce to one, and the transformed
\(\zeta_i\) reduce to \(\overline{\zeta_i}^{\,\epsilon_i}\ne-1\).
Choose \(\epsilon_i=-1\) precisely where \(v(F_i)<0\); then all
transformed \(F_i\) lie in \(\mathcal O\). We work in these coordinates.

Put \(\delta_j=v(1+F_j)\) for every vertex \(j\).
Suppose some \(F_i\) reduces to \(-1\). Choose \(i\) maximizing
\(\delta_i>0\). For every neighbor \(j\),
\[
 v\!\left(\frac{1+F_i}{1+F_j}\right)=\delta_i-\delta_j\ge0.
\]
The quadratic \eqref{eq:edge-quadratic} at \(i\) has integral
coefficients and unit leading and constant coefficients. Its reduction
is \((A+1)^2\). Thus every incident \(A_{e,i}\) is a unit reducing
to \(-1\). Reducing \eqref{eq:aux-vertex} yields
\[
 \overline{\zeta_i}=(-1)^{1-d_i}(-1)^{d_i}=-1,
\]
a contradiction. An isolated vertex satisfies \(F_i=\zeta_i\)
and gives the same conclusion. Hence every \(1+F_i\) is a unit.

Next suppose \(p=v(F_i)>0\). The linear coefficient of
\eqref{eq:edge-quadratic} reduces to one and its reduction is
\(A(A+1)\). A negative valuation of \(A\) would make its quadratic
term uniquely minimal. Thus there are two possible types of incident variable:
a unit reducing to \(-1\), or a variable with positive valuation.
In the second case, writing the equation as
\(A(s_eA+H)=F_i\), where \(\overline H=1\), shows that
\[
 v(A)=p,\qquad \overline{A/F_i}=1.
\]
Let \(k\) be the number of variables of the first type. Taking
valuations in \eqref{eq:aux-vertex} gives
\[
 0=(1-d_i)p+(d_i-k)p=(1-k)p,
\]
so \(k=1\). Dividing each positive-valuation factor by \(F_i\)
in \eqref{eq:aux-vertex} and reducing then gives
\(\overline{\zeta_i}=(-1)^k=-1\), again a contradiction.
Thus every transformed \(F_i\) is a unit. Undoing the spin changes
shows that the original \(F_i\) are units as well and avoid \(-1\)
in the residue field.

Finally, in \eqref{eq:edge-quadratic} all coefficients are integral
and the leading and constant coefficients are units. A root with
negative valuation would make the quadratic term uniquely minimal;
a root with positive valuation would make the constant term uniquely
minimal. Neither is possible. Hence all \(A_{e,i}\) are units.
\end{proof}

\subsection{The limit states and their unique lifts}

Choose \(c_e,\zeta_i\) algebraically independent over \(\C\), let
\(K=\overline{\C(c_e,\zeta_i)}\), and set
\begin{equation}\label{eq:generic-weak-line}
 s_e=1+tc_e.
\end{equation}
The preceding lemma applies to every admissible \(\mathcal K\)-solution.
The edge equations at \(t=0\), with bars omitted, are
\[
 (A-F_i)(B+1)=0,\qquad (B-F_j)(A+1)=0.
\]
As \(F_i,F_j\ne-1\), the two possibilities are
\[
 (A,B)=(F_i,F_j)\quad\text{or}\quad(A,B)=(-1,-1).
\]
Call the second type active, and let \(S\subseteq E\) be the active
edge set. If \(k_i=d_S(i)\), the vertex equations reduce to
\begin{equation}\label{eq:weak-limit-vertex}
 \zeta_i=(-1)^{k_i}F_i^{\,1-k_i}.
\end{equation}
For \(k_i=0\) there is one root \(F_i=\zeta_i\); for \(k_i=1\)
there is none; and for \(k_i\ge2\) there are \(k_i-1\) distinct roots.
They are nonzero and different from \(-1\).
Different active sets give different edge coordinates. Consequently
there are exactly
\begin{equation}\label{eq:state-sum-degree}
 Q(G)=\sum_{S\subseteq E}\wt(S),\qquad
 \wt(S)=\prod_{i\in V}\varphi(d_S(i))
\end{equation}
limit points in these auxiliary coordinates, where \(\varphi(0)=1\)
and \(\varphi(k)=k-1\) for \(k\ge1\).
The active limits \(A=B=-1\) lie outside the original admissible
fiber at zero coupling; they are used only as limits of admissible
branches.

\begin{lemma}\label{lem:unique-weak-lift}
Each of these \(Q(G)\) limit points lifts to one and only one
admissible \(\mathcal K\)-solution. The lift lies in \(K[[t]]\) in the
\((F,A)\) coordinates and is simple over \(K((t))\).
\end{lemma}

\begin{proof}
Use the full edge residuals
\begin{align*}
 E_{e,i}&=(A-F_i)(B+1)+tc_eAB,\\
 E_{e,j}&=(B-F_j)(A+1)+tc_eAB,
\end{align*}
and \(V_i=F_i^{1-d_i}\prod_{e\ni i}A_{e,i}-\zeta_i\).
At a limit point the edge Jacobian in \((A,B)\) is, respectively,
\[
 \begin{pmatrix}1+F_j&0\\0&1+F_i\end{pmatrix}
 \quad\text{or}\quad
 \begin{pmatrix}0&-1-F_i\\-1-F_j&0\end{pmatrix}.
\]
Both are invertible. After eliminating these blocks, the Schur
complement is the diagonal derivative of
\eqref{eq:weak-limit-vertex}, with entries
\[
 (1-k_i)(-1)^{k_i}F_i^{-k_i}\ne0.
\]
The full Jacobian is therefore invertible.
To apply Hensel lifting, multiply each \(V_i\) with \(d_i\ge1\)
by the unit \(F_i^{d_i-1}\). Together with the edge residuals, these
give a polynomial system \(H\) over \(K[[t]]\) with the same roots
near the limit point and an invertible Jacobian there.
Starting from that point, a solution modulo \(t^r\) lifts uniquely
modulo \(t^{r+1}\): its next coefficient solves the linear system
with coefficient matrix equal to the residual Jacobian at \(t=0\).
Completeness of \(K[[t]]\) gives a solution \(x^*\in K[[t]]^{n+2m}\).
This is the simple-root Hensel argument
\citep[Tag~04GM]{stacks2026cohomology}.

Uniqueness also holds in \(\mathcal O\), without a completeness
assumption on that ring. If \(x\in\mathcal O^{n+2m}\) is another
solution with the same reduction, polynomial differences give
\[
 H(x)-H(x^*)=M(x,x^*)(x-x^*),\qquad
 \overline{M(x,x^*)}=J_{\overline H}(\overline{x^*}).
\]
The matrix \(M(x,x^*)\) has entries in \(\mathcal O\) and unit
determinant, so \(x=x^*\). By \cref{lem:weak-boundary}, this covers
every admissible \(\mathcal K\)-solution with the specified limit.

For inactive edges \(A+1\) and \(B+1\) are units. For active edges
the first nonzero coefficients are
\[
 A+1=\frac{tc_e}{1+F_j^{(0)}}+O(t^2),\qquad
 B+1=\frac{tc_e}{1+F_i^{(0)}}+O(t^2).
\]
These are nonzero series. All \(F_i,A_{e,i},1+F_i\) are units, so
\eqref{eq:aux-recovery} gives admissibility over \(K((t))\).
The Jacobian remains invertible there, proving simplicity.
\end{proof}

\begin{theorem}[Generic degree]\label{thm:degree}
For every finite loopless multigraph, with the factorization and
admissibility conditions above,
\[
 N(G)=Q(G).
\]
\end{theorem}

\begin{proof}
The homomorphism
\[
 \C(s_e,\zeta_i)\longrightarrow K((t)),\qquad
 s_e\longmapsto1+tc_e
\]
is injective. Indeed, if a nonzero polynomial \(H\) became zero,
the polynomial identity \(H(1+tc,\zeta)=0\), evaluated at \(t=1\),
would contradict the algebraic independence of \(c,\zeta\).
Choose \(a_e=(1+tc_e)^{1/2}\) with constant coefficient one and
put \(b_i=\zeta_i\prod_{e\ni i}a_e\).
The \(m+n\) quantities \(s,\zeta\) are rational functions of the
\(m+n\) quantities \(a,b\) and are algebraically independent.
Thus \(a,b\) are algebraically independent as well.
The original parameter field therefore embeds into \(K((t))\).

The finite \'etale generic fiber of \cref{prop:bp-family}, after
base change to the algebraically closed field \(\mathcal K\), has exactly
\(N(G)\) points. By \cref{lem:weak-boundary} each has one of the
\(Q(G)\) limits just enumerated, and by \cref{lem:unique-weak-lift}
each limit has exactly one simple admissible lift. This proves
the equality. For the empty graph both the empty message tuple
and the state sum have cardinality one.
\end{proof}

\begin{corollary}\label{cor:elementary-degree}
The generic degree is multiplicative under disjoint union, unchanged
by attaching a leaf or subdividing an edge, equals one on a forest,
and equals two on a connected graph with a single cycle.
\end{corollary}

\begin{proof}
See \cref{supp:elementary-degree}.
\end{proof}

The theorem concerns a nonempty Zariski open set of parameters and says
nothing about special fibers; \cref{sec:strong-field} gives a separate
complete count for every fixed coupling vector with \(a_e^2\ne1\).

%% file: sections/incidence-algebra.tex
\section{The universal edge-incidence ring}
\label{sec:incidence}

The factors in the degree formula have a cohomological meaning:
\(d-1=\dim H^1(\Pp^1,\cO(-d))\) for \(d\geq2\).  We use this to
construct a canonical ring whose entire minimal free resolution is
indexed by the same edge states as the BP degree, following the
geometric method for syzygies \cite{weyman2003syzygies}: resolve an
incidence scheme by its Koszul complex and take derived global sections.
Related constructions for mixed multilinear systems appear in
\cite[Section~2.2]{bender2021mixedmultilinearv1}; here the
graph-specific input is the concentration of each state in a single
cohomological degree, and no resultant hypothesis on the number of
equations is needed.

\subsection{Edge bundles and the state cohomology}
\label{alg:edge-bundles}

Let \(G=(V,E)\) be a finite loopless multigraph.  Put
\(n=|V|\), \(m=|E|\), and let \(c=c(G)\) and \(\tau=\tau(G)\)
be the numbers of connected components and tree components, respectively.
An isolated vertex counts as a tree component.  We use
\[
 \nu=m-n+c,\qquad
 \rho=m-n+\tau,\qquad
 d_G=4m-\rho=3m+n-\tau.
\]
Thus \(\rho\) is the sum of \(\nu(C)-1\) over the components
\(C\) containing a cycle, and is nonnegative.  Set
\[
 X_G=(\Pp^1)^V,
 \qquad
 L_e=\operatorname{pr}_u^*\cO_{\Pp^1}(1)
       \otimes\operatorname{pr}_v^*\cO_{\Pp^1}(1)
       \quad(e=\{u,v\}).
\]
Parallel edges give separate summands.  To keep track of maps which
reverse an edge orientation, introduce the one-dimensional orientation
space
\[
 \operatorname{or}(e)=\bigwedge^2\C^{\{u,v\}},
 \qquad
 F_G=\bigoplus_{e\in E}L_e\otimes\operatorname{or}(e)^\vee.
\]
Choosing an orientation of each edge trivializes these orientation
spaces.  We suppress them in formulas for dimensions and retain their
natural signs when discussing graph maps.  An ordering of the edges
similarly gives the usual signs for exterior products; changing either
choice gives an isomorphic construction.

For \(S\subseteq E\), let \(d_S(v)\) be its degree at \(v\), put
\(V(S)=\{v:d_S(v)>0\}\), and let \(v(S)=|V(S)|\).
Recall the state weight
\[
 \wt(S)=\prod_{v\in V}\varphi(d_S(v)),
 \qquad
 \varphi(0)=1,\quad \varphi(d)=d-1\quad(d\geq1).
\]
We call \(S\) a contributing state if it has no vertex of degree one;
the empty state is included.  Define the bigraded vector space
\begin{equation}
 E_1^{-q,j}(G)=H^j(X_G,\bigwedge^q F_G^\vee).
 \label{alg:first-page}
\end{equation}
The notation records the first page of the hypercohomology spectral
sequence of an edge Koszul complex.  At this point no edge section has
been chosen, so this space and its product depend only on the graph.

\begin{lemma}[State cohomology]
\label{alg:state-cohomology}
The summand of \(\bigwedge^{|S|}F_G^\vee\) indexed by \(S\) has zero
cohomology if \(S\) contains a vertex of degree one.  Otherwise its
cohomology is concentrated in degree \(v(S)\), where it has dimension
\(\wt(S)\).  Consequently,
\begin{equation}
 \dim E_1^{-q,j}(G)
 =\sum_{\substack{S\subseteq E\\ |S|=q,\ v(S)=j}}\wt(S),
 \qquad
 \sum_{q,j}\dim E_1^{-q,j}(G)=Q(G).
 \label{alg:state-dimensions}
\end{equation}
Every nonzero summand satisfies \(0\leq q-j\leq\rho\), and the
upper bound is attained by a contributing state.
\end{lemma}

\begin{proof}
Up to its one-dimensional orientation factor, the state summand is
\[
 \mathop{\boxtimes}_{v\in V}\cO_{\Pp^1}(-d_S(v)).
\]
On \(\Pp^1\), the bundle \(\cO\) has a one-dimensional \(H^0\),
\(\cO(-1)\) has no cohomology, and \(\cO(-d)\) for \(d\geq2\) has
only \(H^1\), of dimension \(d-1\).  The K\"unneth formula gives
the asserted concentration and dimension.  In a nonempty contributing
state, every active vertex has degree at least two.  Thus
\(2|S|\geq2v(S)\), proving \(j\leq q\).

There is no nonempty contributing state on a tree component.  On a
component \(C\) containing a cycle, write \(S_C=S\cap E(C)\).
If \(S_C\ne\varnothing\), its cycle space injects into that of
\(C\), so
\[
 |S_C|-v(S_C)=\nu(S_C)-c(S_C)\leq\nu(C)-1,
\]
where \(c(S_C)\) counts the components of its active support.
The empty state satisfies the same inequality.  Summing over components
gives \(q-j\leq\rho\).

To attain the bound, take the full 2-core of each component containing
a cycle, and no edges in the tree components.  These cores have minimum
degree at least two.  Deleting pendant vertices preserves \(m_C-n_C\),
so their union satisfies \(|S|-v(S)=\rho\) and has positive weight.
For a forest the empty state attains \(\rho=0\).
\end{proof}

Cup product and exterior multiplication make \(E_1(G)\) a bigraded
algebra with total degree \(j-q\); we shall identify it with a Tor algebra.

\subsection{The universal incidence scheme}
\label{alg:universal-incidence}

Let
\begin{equation}
 \mathbb B_G=H^0(X_G,F_G),
 \qquad S_G=\Sym(\mathbb B_G^*).
 \label{alg:coefficient-ring}
\end{equation}
Each edge contributes four independent coefficients, so \(S_G\) is
a polynomial ring in \(4m\) variables, all of degree one.  We also
write \(\mathbb B_G\) for the corresponding affine coefficient
space.  The universal section of \(F_G\) on
\(X_G\times\mathbb B_G\) has a zero scheme
\begin{equation}
 Z_G=\{(x,b):b_e(x)=0\text{ for every }e\in E\}.
 \label{alg:incidence-scheme}
\end{equation}

\begin{definition}
\label{alg:universal-definition}
The universal edge-incidence ring is
\[
 \mathscr U_G=\Gamma(Z_G,\cO_{Z_G}).
\]
Its grading comes from simultaneous scaling of the edge coefficients.
The projection \(\pi_G:Z_G\to\mathbb B_G\) gives it the structure of
a graded \(S_G\)-algebra.
\end{definition}

\begin{lemma}
\label{alg:incidence-vector-bundle}
The scheme \(Z_G\) is the total space of a vector bundle of rank
\(3m\) on \(X_G\).  The projection \(\pi_G\) is proper,
\(\mathscr U_G\) is finite as an \(S_G\)-module, and
\((\mathscr U_G)_0=\C\).
\end{lemma}

\begin{proof}
See \cref{supp:vector-bundle}.
\end{proof}

The ring \(\mathscr U_G\) need not be generated by \(S_G\); its extra
module generators will encode collections of disjoint cycles.

\subsection{The complete minimal resolution}
\label{alg:resolution-section}

Let \(S_G(-q)\) denote a free graded module whose generator has
internal degree \(q\).  The universal Koszul complex is
\begin{equation}
 \cK_G^{-q}
   =\bigwedge^q F_G^\vee\otimes S_G(-q),
 \qquad d=\iota_{s^{\mathrm{univ}}}.
 \label{alg:universal-koszul}
\end{equation}
The differential has cohomological degree one and internal degree zero.
By \cref{alg:incidence-vector-bundle}, this complex resolves
\(\cO_{Z_G}\).

\begin{theorem}[The graph resolution]
\label{alg:resolution}
The derived global sections of \(\cK_G\) are concentrated in degree
zero and equal \(\mathscr U_G\).  A minimal graded free
\(S_G\)-resolution of \(\mathscr U_G\) has length exactly \(\rho\)
and terms
\begin{equation}
 \mathscr F_p
 =\bigoplus_q H^{q-p}(X_G,\bigwedge^qF_G^\vee)\otimes S_G(-q),
 \qquad 0\leq p\leq\rho.
 \label{alg:free-terms}
\end{equation}
In particular,
\begin{equation}
 \beta_{p,q}^{S_G}(\mathscr U_G)
 =\sum_{\substack{S\subseteq E\\ |S|=q,\ |S|-v(S)=p}}\wt(S).
 \label{alg:betti-state-sum}
\end{equation}
There is a canonical isomorphism of bigraded algebras
\begin{equation}
 \Tor_p^{S_G}(\mathscr U_G,\C)_q
 \cong H^{q-p}(X_G,\bigwedge^qF_G^\vee).
 \label{alg:tor-isomorphism}
\end{equation}
On the right, multiplication is cup product followed by exterior
multiplication, with the total-complex signs.
\end{theorem}

\begin{proof}
We first establish concentration, before choosing a finite free model.
Since \(\cK_G\) resolves an ordinary sheaf,
\[
 \mathrm R\Gamma(X_G\times\mathbb B_G,\cK_G)
 \simeq \mathrm R\Gamma(Z_G,\cO_{Z_G})
\]
has no negative cohomology.  On the other hand, its hypercohomology
spectral sequence starts with
\[
 E_1^{-q,j}(G)\otimes S_G(-q).
\]
Every nonzero term satisfies \(j-q\leq0\) by
\cref{alg:state-cohomology}.  The sequence is bounded, so there is no
positive cohomology either.  Its sole cohomology module is therefore
\(\Gamma(Z_G,\cO_{Z_G})=\mathscr U_G\).

To obtain the resolution itself, take the product of the standard
two-chart covers of \(\Pp^1\).  Apply its \v{C}ech complex to each
\(\bigwedge^qF_G^\vee\), initially without the Koszul differential.
As complexes of \(\C\)-vector spaces, these complexes admit
contractions onto their finite-dimensional cohomology.  Write their
direct sum as \((\iota,p,h)\), with \(h\) preserving \(q\),
and choose the convention
\[
 \iota p=1+d_{\mathrm{Cech}}h+hd_{\mathrm{Cech}}.
\]
Extend the contraction over \(S_G\).  Let \(\delta\) be the Koszul
differential, including the totalization sign.  It lowers \(q\) by
one and multiplies by a coefficient of degree one.  Thus \(\delta h\)
is nilpotent: no composition can lower \(q\) more than \(m\) times.
The finite perturbation lemma
\cite[Section~2.2 and Main Perturbation Lemma~2.4]{crainic2004perturbationv1}
transfers the differential to
\begin{equation}
 d_{\mathscr F}
 =p(1-\delta h)^{-1}\delta\iota
 =\sum_{r\geq1}p(\delta h)^{r-1}\delta\iota.
 \label{alg:transferred-differential}
\end{equation}
This is a finite polynomial sum, not a completed power series.  It
defines a square-zero differential and a homotopy equivalence with the
original total complex.  A term with \(r\) copies of \(\delta\)
lowers the edge number by \(r\), lowers the \v{C}ech degree by
\(r-1\), and has coefficient degree \(r\).  Consequently it maps
the term in homological degree \(p\) to the one in degree \(p-1\)
with internal degree preserved.  The underlying modules are precisely
\eqref{alg:free-terms}; their range follows from
\cref{alg:state-cohomology}.

Concentration of the original derived global sections proves exactness
of this finite complex in positive homological degrees.  Every entry
of \(d_{\mathscr F}\) lies in the homogeneous maximal ideal
\((S_G)_+\).  The resolution is therefore minimal, proving
\eqref{alg:betti-state-sum}.  The state attaining \(\rho\) in
\cref{alg:state-cohomology} makes \(\mathscr F_\rho\) nonzero, so
the projective dimension is exactly \(\rho\).
Although the chosen contraction affects
the displayed matrices, it does not affect the minimal resolution up
to graded isomorphism.

For the algebra statement, use the universal Koszul complex with its
unit and exterior product.  Derived base change to the coefficient
origin gives
\begin{equation}
 \mathscr U_G\otimes^{\mathbf L}_{S_G}\C
 \simeq
 \mathrm R\Gamma\!\left(
 X_G,\bigoplus_q\bigwedge^qF_G^\vee[q]\right).
 \label{alg:zero-base-change}
\end{equation}
Indeed, the Koszul differential becomes zero at this origin.  The
bounded \v{C}ech model is flat over \(S_G\), so it also computes
this derived tensor product, and the base-change map of \v{C}ech
complexes is a map of differential graded algebras for the Koszul
products, so the displayed equivalence respects the derived products.
Equivalently, the projection \(p\colon X_G\times\mathbb B_G\to\mathbb B_G\)
is proper, flat and of finite presentation, and \(\cK_G\) is a perfect
complex, being a bounded complex of locally free sheaves of finite
rank; hence \(\mathrm Rp_*\cK_G\) is perfect and its formation commutes
with the base change to the origin
\cite[Lemma~36.30.4, Tag~0B91]{stacks2026cohomology}, and this
base-change map is induced by the same map of \v{C}ech complexes.
Taking cohomology in degree \(-p\) and internal degree \(q\) proves
\eqref{alg:tor-isomorphism} with its product.  This identification
uses the original Koszul algebra and base change, rather than a choice
of matrices in \eqref{alg:transferred-differential}, and is canonical.
\end{proof}

\begin{corollary}[The BP degree as a total Betti number]
\label{alg:total-betti}
For every finite loopless multigraph \(G\),
\begin{equation}
 N(G)=Q(G)=\sum_{p,q}\beta_{p,q}^{S_G}(\mathscr U_G).
 \label{alg:bp-total-betti}
\end{equation}
\end{corollary}

\begin{proof}
Summing \eqref{alg:betti-state-sum} counts each contributing state
once, with weight \(\wt(S)\).  The equality with the generic admissible
BP degree is \cref{thm:degree}.
\end{proof}

\subsection{Normality and Cohen--Macaulayness}
\label{alg:cm-section}

\begin{proposition}
\label{alg:normal-cm}
For every finite loopless multigraph \(G\), the ring
\(\mathscr U_G\) is a normal Cohen--Macaulay domain of dimension
\(d_G=3m+n-\tau\).
\end{proposition}

\begin{proof}
The variety \(X_G\) is smooth and irreducible, so the vector bundle
\(Z_G\to X_G\) is smooth and irreducible as well.  In particular,
its ring of global functions is a domain.  To see that it is integrally
closed, take an element of \(\Frac(\mathscr U_G)\) integral over
\(\mathscr U_G\).  Regard it as a rational function on \(Z_G\).
Its monic equation also makes it integral over every local ring of
\(Z_G\).  These local rings are normal, so the function is regular
everywhere and belongs to \(\Gamma(Z_G,\cO_{Z_G})\).

For a tree component \(T\), the resolution has only its empty-state
term, and hence \(\mathscr U_T=S_T\).  For a component \(C\)
containing a cycle, module finiteness identifies the support with the
proper image \(\pi_C(Z_C)\), giving
\(\dim\mathscr U_C\leq\dim Z_C=3m_C+n_C\).
The incidence scheme of a disjoint union is the product of the
component incidence schemes, and global sections over \(\C\) give
the corresponding tensor product.  Consequently,
\[
 \dim\mathscr U_G
 \leq\sum_{C\text{ cyclic}}(3m_C+n_C)
       +\sum_{T\text{ tree}}4m_T
 =3m+n-\tau=d_G.
\]
Here ``cyclic'' means containing at least one cycle.

By \cref{alg:resolution}, the projective dimension over \(S_G\)
is \(\rho\).  At the homogeneous maximal ideal, the depth
inequality for a finite free resolution gives
\[
 \depth_{(S_G)_+}\mathscr U_G\geq4m-\rho=d_G.
\]
Since depth is bounded by dimension, equality holds throughout.
Thus the prime ideal \(I=\ker(S_G\to\mathscr U_G)\) has height
\(\rho\).

Cohen--Macaulayness at every prime follows, since \(\mathscr U_G\) is a
finitely generated graded module over a polynomial ring; a direct
localization of the resolution is recorded in \cref{supp:cm-localization}.
\end{proof}

\subsection{Graph maps and the Tor product}
\label{alg:functoriality-section}

We make the category of graphs explicit.  A noncontracting graph map
\(\phi:G\to H\) consists of maps on vertices and edges taking the
two distinct endpoints of each edge bijectively onto the endpoints of
its image edge.  Different vertices or edges may have the same image,
but an edge cannot be sent to a vertex.  The induced coordinate map
\[
 a_\phi:X_H\longrightarrow X_G,
 \qquad (x_w)_{w\in V(H)}\longmapsto(x_{\phi(v)})_{v\in V(G)}
\]
goes in the opposite direction.  The endpoint bijections also identify
the orientation lines.  In chosen orientations their signs are
\(\epsilon_\phi(e)\in\{1,-1\}\).

\begin{proposition}[Functoriality]
\label{alg:functoriality}
Every noncontracting graph map induces a canonical graded algebra map
\[
 \mathscr U_\phi:\mathscr U_G\longrightarrow\mathscr U_H,
\]
compatible with the induced coefficient-ring map \(S_G\to S_H\).
These maps preserve identities and composition, and
\[
 \mathscr U_{G\sqcup H}
 \cong\mathscr U_G\otimes_{\C}\mathscr U_H.
\]
The Tor isomorphism \eqref{alg:tor-isomorphism} is natural for these
maps, with changing coefficient rings.
\end{proposition}

\begin{proof}
See \cref{supp:functoriality}.
\end{proof}

The multiplication already has a concrete graph interpretation in
homological degree zero.  In the next statement a simple cycle may
have length two, using two distinct parallel edges.

\begin{proposition}[The cycle-compatibility algebra]
\label{alg:cycle-product}
There is a graded algebra isomorphism
\begin{equation}
 \mathscr U_G/(S_G)_+\mathscr U_G
 \cong
 \C[z_C:C\text{ a simple cycle of }G]
 \big/
 (z_Cz_D:V(C)\cap V(D)\ne\varnothing),
 \qquad \deg z_C=|C|.
 \label{alg:cycle-algebra}
\end{equation}
The relations include \(z_C^2=0\).
\end{proposition}

\begin{proof}
See \cref{supp:cycle-product}.
\end{proof}

\begin{corollary}[Module generators and real sectors]
\label{alg:generator-real-bridge}
For every finite loopless multigraph \(G\),
\[
 \sum_q\beta_{0,q}^{S_G}(\mathscr U_G)
 =\#\{\text{collections of vertex-disjoint simple cycles in }G\}
 \leq2^{\nu(G)},
\]
where the empty collection is included.  Equality holds if and only if
no two edge-disjoint simple cycles meet at a vertex.  In particular,
it holds when the maximum degree is at most three.  Whenever equality
holds, this number of module generators equals the minimum number of
real admissible BP fixed points over external fields, for any fixed
finite nonzero real couplings, as computed in \cref{thm:real}.
\end{corollary}

\begin{proof}
See \cref{supp:generator-bridge}.
\end{proof}

The condition fails as soon as two edge-disjoint cycles share a vertex:
\begin{center}\footnotesize
\begin{tabular}{@{}lccc@{}}
\toprule
Graph & \(Q(G)\) & \(\sum_q\beta_{0,q}^{S_G}(\U_G)\) & \(2^{\nu(G)}\)\\
\midrule
Triangle & 2 & 2 & 2\\
\(K_4\) & 48 & 8 & 8\\
Two triangles sharing one vertex & 6 & 3 & 4\\
\bottomrule
\end{tabular}
\end{center}
For the last graph the union of the two triangles is Eulerian with a
vertex of degree four; it has weight \(3\), lies in homological degree
one, and contributes one real sector but no module generator.

\subsection{The graph polynomials and Artin reductions}
\label{alg:polynomials-section}

We use the graph polynomials of Watanabe and Fukumizu
\cite{watanabe2011graphpolynomials}.
With
\[
 f_0(\gamma)=1,\quad f_1(\gamma)=0,\quad
 f_{d+1}(\gamma)=\gamma f_d(\gamma)+f_{d-1}(\gamma),
\]
their two-variable polynomial is
\[
 \theta_G(t,\gamma)
 =\sum_{S\subseteq E}t^{|S|}\prod_{v\in V}f_{d_S(v)}(\gamma).
\]
For any finite graph, their normalization defines the polynomial
\begin{equation}
 \omega_G(t)=(1-t)^{n-m}\theta_G(t,2\mathrm i),
 \qquad \mathrm i^2=-1.
 \label{alg:omega-definition}
\end{equation}
We retain this normalization when tree components are present.
Both \(\theta_G\) and the factor \((1-t)^{n-m}\) are multiplicative
over connected components, hence so is \(\omega_G\); a tree \(T\) has
\(\theta_T=1\) and \(\omega_T=1-t\).  Therefore
\begin{equation}
 \omega_G(t)=(1-t)^{\tau}\prod_{C}\omega_C(t),
 \qquad
 \romega_G(t):=\frac{\omega_G(t)}{(1-t)^{\tau}}=\prod_{C}\omega_C(t),
 \label{alg:omega-reduced}
\end{equation}
where \(C\) runs over the components containing a cycle.  Thus
\(\romega_G\) is a polynomial, the tree-reduced form of \(\omega_G\),
and \(\romega_G=\omega_G\) when \(\tau=0\).

\begin{proposition}[Three invariants of one resolution]
\label{alg:hilbert}
For every finite loopless multigraph \(G\),
\begin{align}
 \sum_{p,q}(-1)^p\beta_{p,q}^{S_G}(\mathscr U_G)t^q
   &=\theta_G(t,2\mathrm i),
   \label{alg:theta-numerator}\\
 \Hilb_{\mathscr U_G}(t)
   &=\frac{\theta_G(t,2\mathrm i)}{(1-t)^{4m}}
     =\frac{\omega_G(t)}{(1-t)^{3m+n}}.
   \label{alg:universal-hilbert}
\end{align}
For a general sequence of \(d_G=3m+n-\tau\) linear forms from \((S_G)_1\),
the quotient \(A_G\) is a finite-dimensional graded algebra with
\begin{equation}
 \Hilb_{A_G}(t)=\romega_G(t)=\frac{\omega_G(t)}{(1-t)^\tau},
 \qquad \dim_{\C}A_G=\romega_G(1).
 \label{alg:artin-hilbert}
\end{equation}
These linear forms are a regular sequence on \(\mathscr U_G\).
In particular, \(\omega_G\) itself is the Artin Hilbert series when
\(G\) has no tree component.
\end{proposition}

\begin{proof}
See \cref{supp:three-invariants}.
\end{proof}

For graphs without tree components the nonnegativity of the coefficients
of \(\omega_G\) is part of the cited theory, and \eqref{alg:artin-hilbert}
realizes them as graded dimensions; a forest has \(\mathscr U_G=S_G\)
and \(\romega_G=1\).  One-dimensional sections over a polynomial ring in
\(\rho+1\) variables have the same Betti table (\cref{supp:pencil}), and
for connected graphs \(Q(G)\) satisfies the sharp bounds
\((\nu-1)2^\nu+2\le Q(G)\le(4+2\sqrt2)^{\nu-1}+(4-2\sqrt2)^{\nu-1}\)
(\cref{alg:degree-bounds}).

\subsection{The complete graph on four vertices}
\label{alg:k4-section}

For \(G=K_4\), we have \(n=4\), \(m=6\), \(\nu=3\).
The contributing states and their contributions to the resolution and
strong-field root counts are displayed in \cref{fig:k4-states}.
\input{figures/k4-states}

Thus \(S=S_{K_4}\) has 24 variables, \(\mathscr U_{K_4}\) is a
normal Cohen--Macaulay domain of dimension 22, and its full minimal
resolution is
\begin{equation}
 0\longrightarrow S(-6)^{16}
 \longrightarrow S(-5)^{24}
 \longrightarrow S\oplus S(-3)^4\oplus S(-4)^3
 \longrightarrow\mathscr U_{K_4}\longrightarrow0.
 \label{alg:k4-resolution}
\end{equation}
The alternating numerator and Hilbert series are
\begin{align*}
 \theta_{K_4}(t,2\mathrm i)
   &=1+4t^3+3t^4-24t^5+16t^6,\\
 \Hilb_{\mathscr U_{K_4}}(t)
   &=\frac{1+2t+3t^2+8t^3+16t^4}{(1-t)^{22}},
\end{align*}
so the Artin reduction has Hilbert series
\(\omega_{K_4}(t)=1+2t+3t^2+8t^3+16t^4\), the polynomial computed by
Watanabe and Fukumizu \cite[Example~2]{watanabe2011graphpolynomials},
and length 30, while the total Betti number is \(8+24+16=48=N(K_4)\),
the sharp upper bound of \cref{alg:degree-bounds}.  The \(8=2^3\)
module generators are the unit, four triangle classes, and three
four-cycle classes; any two cycles of \(K_4\) meet at a vertex, so every
product of two nonunit classes vanishes in
\(\mathscr U_{K_4}/S_+\mathscr U_{K_4}\).

An independent saturated Gr\"obner-basis computation at a recorded
specialization of the $K_4$ parameters over $\F_{32003}$ gives a
zero-dimensional quotient of length $48$, agreeing with this state
sum.\footnote{The input script and the saved Singular output are
reproduced in \cref{supp:k4-groebner}.
This is a finite-characteristic specialization check with multiplicity;
it does not certify $48$ rational points or simplicity.}
The proof of $N=Q$ above is independent of this computation.

Together, \cref{alg:resolution,alg:normal-cm,alg:functoriality,alg:hilbert}
and \cref{thm:degree} establish the algebraic assertions of
\cref{thm:structure}.

%% file: figures/k4-states.tex
\begin{figure}[htbp]
\centering
\begingroup
\newcommand{\kfourstate}[1]{%
  \begin{tikzpicture}[x=0.5cm,y=0.5cm,baseline=-0.15cm]
    \coordinate (a) at (0.75,1.25);
    \coordinate (b) at (0,0);
    \coordinate (c) at (1.5,0);
    \coordinate (d) at (0.75,0.42);
    \foreach \u/\v in {a/b,a/c,a/d,b/c,b/d,c/d}
      \draw[black!35,densely dotted,line width=0.4pt] (\u)--(\v);
    #1
    \foreach \u in {a,b,c,d} \fill (\u) circle[radius=1.25pt];
  \end{tikzpicture}%
}
\newcommand{\selectededge}[2]{\draw[MidnightBlue,line width=1.1pt] (#1)--(#2);}
\setlength{\tabcolsep}{6pt}
\footnotesize
\begin{tabular}{@{}lccccc@{}}
\toprule
State type & Empty & Triangle & Four-cycle & \(K_4-e\) & \(K_4\)\\[3pt]
 & \kfourstate{}
 & \kfourstate{\selectededge{a}{b}\selectededge{b}{c}\selectededge{c}{a}}
 & \kfourstate{\selectededge{a}{b}\selectededge{b}{d}\selectededge{d}{c}\selectededge{c}{a}}
 & \kfourstate{\selectededge{a}{b}\selectededge{a}{c}\selectededge{b}{c}\selectededge{b}{d}\selectededge{c}{d}}
 & \kfourstate{\selectededge{a}{b}\selectededge{a}{c}\selectededge{a}{d}\selectededge{b}{c}\selectededge{b}{d}\selectededge{c}{d}}\\[5pt]
Number of states & 1 & 4 & 3 & 6 & 1\\
Weight \(\wt(S)\) & 1 & 1 & 1 & 4 & 16\\
Internal degree \(q=|S|\) & 0 & 3 & 4 & 5 & 6\\
Homological degree \(p=|S|-v(S)\) & 0 & 0 & 0 & 1 & 2\\
Complex roots, total & 1 & 4 & 3 & 24 & 16\\
Real roots, total & 1 & 4 & 3 & 0 & 0\\
\bottomrule
\end{tabular}
\endgroup
\caption{Contributing state types of \(K_4\), with the edges of \(S\)
solid. Root counts are for sufficiently strong fields at fixed nonzero
real couplings; the first three columns give the eight module generators and
the eight Eulerian sectors.}
\label{fig:k4-states}
\end{figure}

%% file: sections/strong-field.tex
\section{The strong-field boundary of the BP family}
\label{sec:strong-field}

The factors in the state sum have a second geometric interpretation in
the actual BP family.  Fix edge parameters
\begin{equation}
  a_e\in\C^*,\qquad a_e^2\ne1\quad(e\in E),
  \qquad p_i=b_i^{-1}\quad(i\in V).
  \label{sf:parameters}
\end{equation}
The edge parameters need not be generic.  We study the boundary where
all $p_i$ vanish, with no restriction on their relative rates.
Call $S\subseteq E$ a \emph{state} if $k_i:=d_S(i)$ is never one, and put
\begin{equation}
 r_i=\begin{cases}1,&k_i=0,\\ k_i-1,&k_i\ge2,\end{cases}
 \qquad \wt(S)=\prod_{i\in V}r_i.
 \label{sf:state-weights}
\end{equation}
Thus $Q(G)=\sum_S \wt(S)$.  A vertex with $k_i\ge2$ will be called active
in this section.

\begin{theorem}[Statewise power-map model]
\label{sf:normal-form}
For each state $S$ there is a smooth local analytic extension of the
fixed-$a$ admissible BP family across $p=0$ whose parameter map, in local
coordinates $w=(w_i)_{i\in V}$, is exactly
\begin{equation}
       p_i=(-1)^{k_i}w_i^{r_i}\qquad(i\in V).
       \label{sf:power-map}
\end{equation}
For all sufficiently small nonzero $p_i$, this chart gives $\wt(S)$
distinct simple admissible BP solutions.  The solutions supplied by
different states are distinct.
\end{theorem}

\subsection{Two edge branches near the boundary}

Keep $S$ fixed.  Set $\sigma_i=-1$ at active vertices and $\sigma_i=1$
at inactive vertices, and introduce coordinates
\begin{equation}
 F_i=\begin{cases}f_i,&k_i\ge2,\\ f_i^{-1},&k_i=0.\end{cases}
 \label{sf:vertex-chart}
\end{equation}
Initially all $f_i$ are nonzero.  For $e=\{i,j\}$, the edge equations,
with the full vertex ratios $F_i,F_j$ treated as independent variables,
are
\begin{equation}
 y_{i\to j,e}
 =\frac{a_eF_i+y_{j\to i,e}}{F_i+a_ey_{j\to i,e}},
 \qquad
 y_{j\to i,e}
 =\frac{a_eF_j+y_{i\to j,e}}{F_j+a_ey_{i\to j,e}}.
 \label{sf:edge-equations}
\end{equation}
At an inactive vertex the first fraction can be rewritten as
$(a_e+f_iy_{j\to i,e})/(1+a_ef_iy_{j\to i,e})$; at an active vertex
it is $(a_ef_i+y_{j\to i,e})/(f_i+a_ey_{j\to i,e})$.
These expressions extend holomorphically to $f=0$ at
\begin{equation}
             y^+_{i\to j,e}(0)=a_e^{\sigma_i}.
             \label{sf:base-branch}
\end{equation}
At this point, each fraction has zero derivative with respect to the
opposite message.  Consequently the Jacobian of the two residuals
``message minus fraction'' with respect to the two messages is the
identity.  The analytic implicit function theorem gives unique
nonvanishing holomorphic functions $y^+_{i\to j,e}(f)$ near zero.
They may be constructed simultaneously for all edges, and also with
$a$ varying near its fixed value.  The superscript $+$ denotes the
branch which is positive when $a_e>0$ and all $f_i>0$; it imposes no
order on complex points.

If $e\in S$, both its endpoints are active.  The other branch is
\begin{equation}
 y^-_{i\to j,e}=-\frac{F_j}{y^+_{i\to j,e}},
 \qquad
 y^-_{j\to i,e}=-\frac{F_i}{y^+_{j\to i,e}}.
 \label{sf:other-branch}
\end{equation}
It satisfies the same edge equations.  For an explicit verification,
write $a=a_e$, $y_i=y_{i\to j,e}$, $y_j=y_{j\to i,e}$, and
\[
 R_i(y)=y_i(F_i+ay_j)-aF_i-y_j,
 \qquad R_j(y)=y_j(F_j+ay_i)-aF_j-y_i.
\]
The involution
$\iota(y_i,y_j)=(-F_j/y_i,-F_i/y_j)$ satisfies
\begin{equation}
 R_i(\iota y)=-\frac{F_i}{y_iy_j}R_j(y),\qquad
 R_j(\iota y)=-\frac{F_j}{y_iy_j}R_i(y).
 \label{sf:edge-involution}
\end{equation}
For nonzero $F_i,F_j,y_i,y_j$ this is an invertible change of variables
and an invertible change of generators of the residual ideal.  In
particular, the edge Jacobian remains invertible on the other branch.
Moreover, $y^-_{i\to j,e}\to0$, whereas the base-branch messages tend
to nonzero constants.

\subsection{Absorbing the analytic units}

Choose the message branch $y^-$ on $S$ and $y^+$ off $S$, and write
$\varepsilon_e=-1$ on $S$ and $\varepsilon_e=1$ off $S$.
The vertex consistency equations become
\begin{equation}
 F_i=b_i(-1)^{k_i}F_i^{k_i}
       \prod_{e=\{i,j\}\ni i}(y^+_{j\to i,e})^{\varepsilon_e}.
 \label{sf:consistency}
\end{equation}
Define the analytic units
\begin{equation}
 C_{S,i}(f)=
   \prod_{e=\{i,j\}\ni i}(y^+_{j\to i,e}(f))^{-\varepsilon_e},
 \qquad
 C_{S,i}(0)=\prod_{e=\{i,j\}\ni i}a_e^{-\varepsilon_e\sigma_j}\ne0.
 \label{sf:units}
\end{equation}
For both types of vertices, \eqref{sf:consistency} is equivalent to
\begin{equation}
                   f_i^{r_i}=(-1)^{k_i}p_iC_{S,i}(f).
                   \label{sf:unit-equations}
\end{equation}
On a small simply connected neighborhood, choose an $r_i$th root of
each unit and set
\begin{equation}
                    w_i=f_iC_{S,i}(f)^{-1/r_i}.
                    \label{sf:coordinate-change}
\end{equation}
Its derivative at zero is diagonal with nonzero diagonal entries:
derivatives of the units in other variables are multiplied by $f_i$.
It is therefore an analytic coordinate change, even though the units
depend on all the variables.  Equation \eqref{sf:unit-equations} now
becomes \eqref{sf:power-map} exactly.

For all sufficiently small nonzero $p_i$, the independent power
equations give $\wt(S)$ distinct points, all with $f_i\ne0$.
We check admissibility in the original coordinates.  All reconstructed
messages are nonzero.  With
$X_{i,e}=F_i/y_{j\to i,e}$, an inactive vertex has
$X_{i,e}\to\infty$, an active vertex on an edge outside $S$ has
$X_{i,e}\to0$, and on an edge in $S$ one has
\[
             X_{i,e}=-y^+_{j\to i,e}\longrightarrow-a_e^{-1}.
\]
In all three cases $X_{i,e}+a_e\ne0$ in a sufficiently small
punctured neighborhood; in the last case this uses $a_e^2\ne1$.
Also $1+F_i\ne0$, since $F_i\to0$ or $F_i\to\infty$.
Thus every constructed point belongs to the specified admissible
localization, rather than only to the equations after clearing
denominators.

To check simplicity, adjoin $F_i$ and the equations
$F_i=b_i\prod_{j\to i,e}y_{j\to i,e}$ to the original BP system.
The edge Jacobians are invertible on the base branches and, by
\eqref{sf:edge-involution}, on the chosen other branches as well.
Eliminating the edge messages leaves precisely
\eqref{sf:unit-equations}.  At the constructed points the coordinate
change \eqref{sf:coordinate-change} is invertible, and the derivative
of \eqref{sf:power-map} is the invertible diagonal matrix
$\diag((-1)^{k_i}r_iw_i^{r_i-1})$.
The Schur complement therefore shows that the original BP Jacobian
with respect to the message variables is invertible.
Finally, the messages which tend to zero are exactly those on $S$,
in both orientations.  Different states give disjoint sets of
solutions after the neighborhood is made smaller.
This proves Theorem~\ref{sf:normal-form}.

\subsection{Exhaustion of the complex fiber}

The construction has produced $Q(G)$ simple points for every fixed $a$
satisfying \eqref{sf:parameters}; the following lemma excludes
additional points in such a special fiber.

\begin{lemma}[Saturation of the generic degree]
\label{sf:degree-saturation}
Let $\pi:X\to P$ be a dominant morphism of smooth irreducible complex
varieties of the same dimension.  Suppose that $\pi^{-1}(U)\to U$ is
finite \textup{\'{e}}tale of degree $N$ for a nonempty Zariski open
$U\subseteq P$.  If one fiber contains $N$ distinct points at which
$d\pi$ is invertible, then that fiber consists of precisely these
points and is reduced.
\end{lemma}

\begin{proof}
The inverse function theorem gives disjoint neighborhoods of the
$N$ points, each mapping biholomorphically onto a common analytic
neighborhood $B$ of their image $p$.  If an additional point $x$ lay
over $p$, choose a neighborhood of $x$ disjoint from these $N$
neighborhoods and contained in $\pi^{-1}(B)$.
The nonempty Zariski open set $\pi^{-1}(U)$ is analytically dense in
the irreducible smooth variety $X$.  The neighborhood of $x$ thus
contains a point $x'\in\pi^{-1}(U)$.
Its image belongs to $B\cap U$ and has at least $N+1$ distinct
preimages, a contradiction.  At each of the remaining $N$ points
$\pi$ is a local isomorphism, so its fiber is reduced there.
\end{proof}

\begin{corollary}[All complex roots at strong field]
\label{sf:exhaustion}
For every fixed choice \eqref{sf:parameters}, there is $\epsilon>0$
such that whenever $0<|p_i|<\epsilon$ for all $i$, the admissible BP
fiber has exactly $Q(G)$ points and every point is simple.
In particular it has no additional isolated points, multiple points,
or positive-dimensional components.
\end{corollary}

\begin{proof}
Apply Lemma~\ref{sf:degree-saturation} to the full parameter family
of Proposition~\ref{prop:bp-family}.  By Theorem~\ref{thm:degree}, its
generic degree is $Q(G)$.
The $\sum_S \wt(S)=Q(G)$ points of Theorem~\ref{sf:normal-form} are
noncritical points of the full parameter projection: invertibility
of the message Jacobian permits all parameters, including $a$, to
vary by the implicit function theorem.  The lemma applies to each
of the stated fibers.
\end{proof}

\subsection{Boundary algebras and state factors}

\begin{proposition}[The algebra at a state]
\label{sf:boundary-algebra}
The special fiber of the state chart at $p=0$ has local algebra
\begin{equation}
 \mathcal B_{S,0}=\C[w_i:i\in V]/(w_i^{r_i}:i\in V).
 \label{sf:local-algebra}
\end{equation}
It is a complete intersection of length $\wt(S)$.
The ramification index in the $i$th boundary direction is $r_i$.
\end{proposition}

\begin{proof}
See \cref{supp:boundary-algebra}.
\end{proof}

The state summand of the cohomology table has the factor
$H^1(\Pp^1,\cO(-k_i))$, of dimension $k_i-1=r_i$, at an active vertex
and $H^0(\Pp^1,\cO)$ at an inactive one, so the dimensions entering the
Betti table are the local lengths and ramification indices of the BP
family.  The comparison is statewise and linear, identifying neither
cup product with the multiplication in \eqref{sf:local-algebra} nor a
deformation differential between states; $p=0$ lies outside the
admissible BP space, so this is boundary ramification
(\cref{sf:serre-comparison} gives a formulation through Serre duality).

%% file: sections/real-sectors.tex
\section{Real sign sectors and the minimum real count}
\label{sec:real}

We now take physical parameters
$a_e=e^{2J_e}>0$ and $b_i=e^{2h_i}>0$.
A real solution is still required to lie in the admissible
localization, but its messages may be negative.
Write $R_G(J,h)$ for the number of distinct admissible real
solutions, allowing the value $\infty$, and $N_+(J,h)$ for the
number with all messages positive.
The binary cycle space controls a decomposition of the entire
real solution set, not just the strong-field limit.

\subsection{A maximum principle for the vertex ratios}

\begin{proposition}[Eulerian message signs]
\label{real:sign-sectors}
At any admissible real solution with $a_e,b_i>0$, one has $F_i>0$
at every vertex.  The two messages on each edge have the same sign,
and the set of edges with negative messages is Eulerian.
These conclusions require neither nondegeneracy nor a finite fiber.
\end{proposition}

\begin{proof}
Put $m_i=(F_i-1)/(F_i+1)$, which is defined by admissibility.
Suppose that some $F_i<0$, and choose such an $i$ with maximal
$|m_i|$.  Since $b_i>0$, at least one incoming message
$y_{j\to i,e}$ is negative.  Write
\[
 X=X_{j,e},\qquad Y=X_{i,e},\qquad a=a_e,
 \qquad Z=aXY+X+Y+a.
\]
The update $(aX+1)/(X+a)<0$ implies $X<0$.
Since $F_i=Yy_{j\to i,e}<0$, it follows that $Y>0$.
The reverse message is therefore positive and
$F_j=Xy_{i\to j,e}<0$.
The edge identities give
\[
 1+F_j=\frac{Z}{Y+a},\qquad
 1+F_i=\frac{Z}{X+a}.
\]
All denominators and $Z$ are nonzero by admissibility.  Using
$m(F)^2=1-4F/(1+F)^2$, direct subtraction yields
\begin{equation}
 m_j^2-m_i^2
 =\frac{4a(X-Y)(XY-1)}{Z^2}>0,
 \label{real:maximum-principle}
\end{equation}
because $X<0<Y$.  This contradicts the choice of $i$, so all $F_i$
are positive.  An isolated vertex already has $F_i=b_i>0$.

If $y_{i\to j,e}<0$, its cavity ratio $X_{i,e}$ is negative.
The identity $F_i=X_{i,e}y_{j\to i,e}>0$ then forces the reverse
message to be negative as well.  At each vertex the product of
the incoming signs equals the sign of $F_i/b_i$, hence is positive.
Every vertex is therefore incident to an even number of negative
edges.  Such subsets are precisely the elements of
$H_1(G;\F_2)$.
\end{proof}

In particular, the positive solutions form the sector
$S=\varnothing$, and messages on bridges are always positive.
The positivity of the vertex ratios in the proposition does not
assert that all edge beliefs of a nonpositive solution are
probability distributions.

\subsection{Two branches on an edge}

For the rest of this section assume $J_e\ne0$ on every edge.
We first treat the full ratios $F_i,F_j>0$ as independent variables.

\begin{lemma}[Real edge branches]
\label{real:edge-branches}
For an edge $e=\{i,j\}$ with $a_e>0$, $a_e\ne1$, and prescribed
$F_i,F_j>0$, the edge equations have exactly two real admissible
solutions: one with both messages positive, and one with both
messages negative.  Both depend real analytically on $(F_i,F_j)$,
and their edge Jacobians are invertible.
Write $y_{j\to i,e}^{\pm}$ for the incoming message on the
respective branch. Then
\begin{equation}
             y_{j\to i,e}^{+}y_{j\to i,e}^{-}=-F_i.
             \label{real:edge-product}
\end{equation}
\end{lemma}

\begin{proof}
See \cref{supp:edge-branches}.
\end{proof}

Set $F_i=e^{2u_i}$ and
$z_{e,i}^{+}(u)=\tfrac12\log y_{j\to i,e}^{+}$.
On the negative branch \eqref{real:edge-product} gives
\begin{equation}
 \tfrac12\log|y_{j\to i,e}^{-}|=u_i-z_{e,i}^{+}(u).
 \label{real:log-branches}
\end{equation}
On the positive branch the cavity ratios are positive, so the
positive BP update lies strictly between $a_e$ and $a_e^{-1}$.
Consequently
\begin{equation}
                   |z_{e,i}^{+}(u)|<|J_e|
                   \qquad(u\in\R^V).
                   \label{real:edge-bound}
\end{equation}

\subsection{The global vertex map}

Fix an Eulerian subset $S$.  Write $k_i=d_S(i)$,
$D_S=\diag(1-k_i)$, and again let $\varepsilon_e=-1$ on $S$ and
$\varepsilon_e=1$ off $S$.  Define
\begin{equation}
 \mathfrak h_{S,i}(u)
       =(1-k_i)u_i-\sum_{e\ni i}\varepsilon_e z_{e,i}^{+}(u).
 \label{real:field-map}
\end{equation}
Every diagonal entry of $D_S$ is nonzero because every $k_i$ is even.

\begin{proposition}[Vertex coordinates on a real sector]
\label{real:vertex-map}
The admissible real BP solutions with negative-edge support $S$
are in one-to-one correspondence with
\begin{equation}
                    \mathfrak h_S(u)=h,\qquad u\in\R^V.
                    \label{real:vertex-equation}
\end{equation}
The correspondence preserves the local real analytic solution
spaces and nondegeneracy.  Every such solution satisfies
\begin{equation}
 |(1-k_i)u_i-h_i|\le\bar J_i,
 \qquad \bar J_i=\sum_{e\ni i}|J_e|.
 \label{real:global-box}
\end{equation}
\end{proposition}

\begin{proof}
See \cref{supp:vertex-map}.
\end{proof}

\begin{proposition}[Every sector persists for all fields]
\label{real:sector-degree}
For fixed finite nonzero $J$, the map $\mathfrak h_S:\R^V\to\R^V$
is proper and has Brouwer degree
\begin{equation}
                  \deg\mathfrak h_S=(-1)^{v(S)},
                  \label{real:degree}
\end{equation}
where $v(S)$ is the number of vertices incident to $S$.
In particular every Eulerian sector contains an admissible real
BP solution for every $h\in\R^V$, and
\begin{equation}
                         R_G(J,h)\ge2^{\nu(G)}.
                         \label{real:lower-bound}
\end{equation}
\end{proposition}

\begin{proof}
If $|V|\leq1$, the loopless graph has no edges and the map is the
identity on $\R^V$, of degree one. Assume henceforth that $|V|\geq2$.
Write $\mathfrak h_S(u)=D_Su-\Psi_S(u)$.
Equation~\eqref{real:edge-bound} gives
$|\Psi_{S,i}(u)|\le\bar J_i$ for all $u$.
Since $|1-k_i|\ge1$, the homotopy
\[
                   H_t(u)=D_Su-t\Psi_S(u),\qquad 0\le t\le1,
\]
satisfies
$\|H_t(u)\|_\infty\ge\|u\|_\infty-\max_i\bar J_i$ uniformly in $t$.
It is therefore a proper homotopy from the invertible linear map
$D_S$ to $\mathfrak h_S$. For every fixed $h$, all zeros of
$H_t-h$, uniformly in $t$, lie in one sufficiently large ball.
The normalized maps $(H_t-h)/\|H_t-h\|$ on its boundary form a
homotopy of spheres. At $t=0$, removing the translation by $h$
gives the sphere map induced by $D_S$, of degree $\sgn\det D_S$.
Homotopy invariance for compact oriented manifolds
\cite[\S5, Theorem B]{milnor1965differentiable} therefore gives
$\deg\mathfrak h_S=\sgn\det D_S$.
Exactly the $v(S)$ active vertices have negative diagonal entries,
which proves \eqref{real:degree}.
If $\mathfrak h_S-h$ had no zero, its normalization would extend
from the boundary over the whole ball and would have degree zero
\cite[\S5, Lemma 1]{milnor1965differentiable}, a contradiction.
Proposition~\ref{real:vertex-map} supplies a root
in each sector.  The sectors are disjoint and
$|H_1(G;\F_2)|=2^{m-n+c(G)}$, proving \eqref{real:lower-bound}.
\end{proof}

\begin{corollary}[Odd sector counts]
\label{real:sector-parity}
Fix finite nonzero $J$ and an Eulerian $S$. If every admissible real
root with negative-edge support $S$ is nondegenerate, then these roots
are finite in number and that number is odd. In particular, whenever all
positive roots are nondegenerate, their number is odd.
\end{corollary}

\begin{proof}
By \cref{real:vertex-map} the hypothesis says that $h$ is a regular
value of $\mathfrak h_S$. Its fiber is discrete and, by properness,
finite, and the degree of $\mathfrak h_S-h$ on a large ball equals the
sum of $\sgn\det D\mathfrak h_S$ over the fiber
\cite[\S6]{milnor1965differentiable}. This sum is $\pm1$ by
\eqref{real:degree}, so it has an odd number of terms. The degree
concerns the vertex-field map and asserts nothing about the stability
of the BP iteration.
\end{proof}

\subsection{Attainment at strong fields}

\begin{proposition}[The strong-field real chamber]
\label{real:strong-field-count}
Fix finite nonzero real couplings $J$ and field signs
$\eta_i\in\{\pm1\}$.  If every $\eta_i h_i$ is sufficiently large,
then the admissible complex fiber consists of $Q(G)$ simple points.
Exactly $2^{\nu(G)}$ of them are real, one in each Eulerian sector,
and exactly one has all messages positive.
\end{proposition}

\begin{proof}
First take $h_i\to+\infty$ at every vertex, so $p_i>0$ tends to zero.
For $a_e>0$, each unit $C_{S,i}(0)$ in \eqref{sf:units} is positive.
Choose its root in \eqref{sf:coordinate-change} to be real and
positive at zero.  The coordinate change and its inverse then
commute with complex conjugation.
If some $k_i$ is odd, $r_i=k_i-1$ is even and
\eqref{sf:power-map} requires $w_i^{r_i}=-p_i<0$;
that state has no real solutions.
If every $k_i$ is even, every $r_i$ is odd and the power equations
have exactly one real solution, with all $w_i>0$.
The inverse coordinate change has positive unit factors on a small
real neighborhood, so all $f_i>0$ at that solution.
The selected edge branches are consequently positive off $S$ and
negative on $S$.

A solution in one of these charts is real if and only if its
$f$-coordinates, and hence its $w$-coordinates, are real: the full
ratios are recovered from the messages and the real parameters,
and the reconstruction is real analytic.
Corollary~\ref{sf:exhaustion} has already shown that the charts
exhaust the complex fiber.
Thus the real roots are exactly the one root from each Eulerian
state, and the unique positive root is the state $S=\varnothing$.
All are simple by the same corollary.

For arbitrary field signs use the spin change
\begin{equation}
 J'_e=\eta_i\eta_jJ_e\quad(e=\{i,j\}),\qquad
 h'_i=\eta_i h_i,\qquad
 y'_{j\to i,e}=y_{j\to i,e}^{\eta_i},\qquad
 F'_i=F_i^{\eta_i}.
 \label{real:spin-gauge}
\end{equation}
To check the BP equations directly, write
$g_a(X)=(aX+1)/(X+a)$ and observe that
$g_{a^{st}}(X^s)=g_a(X)^t$ for $s,t\in\{\pm1\}$.
The transformed cavity is $X_{i,e}^{\eta_i}$, which proves the
update identity.  New numerator and denominator factors are old
ones multiplied by units or interchanged; both are nonzero on
the admissible set.  Also $1+F_i^{-1}=(1+F_i)/F_i$.
Thus \eqref{real:spin-gauge} is an analytic isomorphism of the
complex admissible fibers and preserves simplicity.
It preserves real message signs as well.  Applying the all-positive
field case to $(J',h')$ proves the proposition.
\end{proof}

\begin{remark}[Generator states are the all-real states]
\label{real:all-real-states}
Fix $a_e>0$, $a_e\ne1$, and small $p_i>0$ as in the proof of
\cref{real:strong-field-count}. A contributing state $S$ has exactly one
real root among its $\wt(S)$ complex roots if every $d_S(i)$ is even
and none otherwise, so all its roots are real if and only if every
$d_S(i)\in\{0,2\}$, that is, if and only if $S$ indexes a module
generator in \cref{alg:cycle-product}. Sending such a state to its
unique root is a canonical injection from the generator states into the
real roots, bijective exactly under the condition of
\cref{alg:generator-real-bridge}; the other real roots come from
Eulerian states with a vertex of degree at least four, each with
$\wt(S)\geq3$ complex roots of which one is real (\cref{sf:k4-states}).
The spin change \eqref{real:spin-gauge} transports this to every field
orthant.
\end{remark}

\begin{proof}[Proof of Theorem~\ref{thm:real}]
The statewise analytic model and its local lengths are
Theorem~\ref{sf:normal-form} and Proposition~\ref{sf:boundary-algebra};
Corollary~\ref{sf:exhaustion} gives completeness of the complex
count.  For each fixed finite nonzero real $J$,
Proposition~\ref{real:sector-degree} gives the lower bound
$R_G(J,h)\ge2^{\nu(G)}$ for every $h$, without a nondegeneracy
assumption.  Proposition~\ref{real:strong-field-count} attains this
bound in every sufficiently strong-field orthant, with
$(N_{\C},R_G,N_+)=(Q,2^{\nu(G)},1)$ and all roots simple.
Therefore
\begin{equation}
              \min_{h\in\R^V}R_G(J,h)=2^{\nu(G)}.
              \label{real:minimum}
\end{equation}
\end{proof}

\begin{corollary}[A degree budget for positive solutions]
\label{real:positive-bound}
At a physical parameter point with all $J_e\ne0$ in the generic
finite \textup{\'{e}}tale locus of the BP family,
\begin{equation}
  2^{\nu(G)}\le R_G(J,h)\le Q(G),\qquad
  1\le N_+(J,h)\le Q(G)-2^{\nu(G)}+1.
  \label{real:positive-budget}
\end{equation}
Moreover
\begin{equation}\label{real:parity}
 R_G(J,h)\equiv Q(G)\equiv2^{\nu(G)}\pmod2.
\end{equation}
\end{corollary}

\begin{proof}
See \cref{supp:positive-bound}.
\end{proof}

Thus generic real counts lie among the integers between $2^{\nu(G)}$
and $Q(G)$ with the parity of $Q(G)$, a class containing both
endpoints; not every such integer need occur.

The nonzero-coupling hypothesis is essential for the bound with the cycle
rank of \(G\): an edge with \(J_e=0\) carries the messages \(1\) on the
admissible real set and can be deleted without changing that set, so the
lower bound then uses the cycle rank of the support graph
(\cref{supp:zero-couplings}).

%% file: sections/discussion.tex
\section*{Discussion}
\addcontentsline{toc}{section}{Discussion}

The two constructions organize the same state sum at different levels:
the edge-incidence ring retains degrees, products, and graph maps, while
the strong-field charts realize each state weight as the length of a BP
boundary algebra and decide which states have real branches;
\cref{alg:generator-real-bridge,real:all-real-states} locate the
discrepancy between module generators and real branches in the Eulerian
states with a vertex of degree at least four. Several questions remain:
graph descriptions of the higher Tor products and of the resolution
differentials; a comparison with the BP boundary that respects
multiplication and deformations (\cref{sf:serre-comparison});
naturality for edge contractions; and the largest generic real and
positive counts at fixed \(J\), with the chambers realizing the
intermediate values allowed by \cref{real:positive-bound}.

%% file: sections/declarations.tex
\section*{Declarations}
Large language model tools were used in drafting, reorganizing, and
checking the proofs and the text of this paper and in preparing the
computational checks cited in \cref{sec:incidence}. The author assumes
responsibility for all content.

%% file: supplement/supplement.tex

\section{Notation}
\label{supp:notation}

The following table collects the recurring notation of the main text.
\begin{center}
\begin{tabular}{@{}ll@{}}
\toprule
Symbol & Meaning\\
\midrule
\(\wt(S),Q(G)\) & State weight and its sum\\
\(\theta_G,\omega_G,\romega_G\) & Graph polynomials and the tree-reduced form \(\omega_G/(1-t)^{\tau}\)\\
\(\beta_{p,q}^{S_G}(\U_G)\) & Graded Betti numbers over the coefficient ring\\
\(a_e,b_i;\ s_e,\zeta_i\) & BP parameters and their rescaled versions in \cref{sec:degree}\\
\(y_{i\to j,e},F_i\) & Oriented message and full vertex ratio\\
\(\U_G,A_G,\mathcal B_{S,0}\) & Edge-incidence ring, Artin section, boundary local algebra\\
\(w_i,r_i\) & Boundary coordinate and its power-map exponent\\
\(N(G),R_G(J,h),N_+(J,h)\) & Generic complex, distinct real, and positive root counts\\
\bottomrule
\end{tabular}
\end{center}

\section{Supplementary material for the degree theorem}
\label{supp:degree-section}

\subsection{Elementary properties of the degree}
\label{supp:elementary-degree}

\begin{proof}[Proof of \cref{cor:elementary-degree}]
The state sum factors over connected components. A selected leaf edge
has a degree-one endpoint and contributes zero. At the new vertex
of a subdivided edge, only selecting both new edges or neither has
nonzero weight, and both local weights are one. Removing leaves
therefore reduces a unicyclic component to a cycle, where only the
empty and full edge sets contribute.
\end{proof}

\subsection{The non-normalizable locus}
\label{supp:nonnormalizable}

Let \(G\) be connected and consider solutions of the rational BP
equations with nonzero messages and nonzero update denominators, but
without the normalization condition. The identity
\((1+F_i)(X_{j,e}+a_e)=(1+F_j)(X_{i,e}+a_e)\) from the proof of
\cref{prop:bp-family} shows that \(1+F_i=0\) forces \(1+F_j=0\) across
every edge; on a connected graph either all \(1+F_i\) are nonzero or
all \(F_i\) equal \(-1\). \Cref{ex:triangle-admissible} exhibits the
second case for a triangle.

\begin{proposition}
\label{supp:nonnormalizable-dim}
Let \(G\) be connected with \(a_e^2\ne1\) for all \(e\). For generic
\((a,b)\), the locus of solutions with \(F_i=-1\) at every vertex is
empty if \(G\) is a tree and is a nonempty smooth variety of pure
dimension \(\nu(G)-1\) if \(\nu(G)\ge1\).
\end{proposition}

\begin{proof}
Orient each edge \(e=i\to j\) and put \(x_e=y_{j\to i,e}\). If
\(F_i=F_j=-1\), the cavity ratios on \(e\) are \(X_{i,e}=-1/x_e\) and
\(X_{j,e}=-1/y_{i\to j,e}\), and the update equation for \(y_{i\to j,e}\)
reads \(y_{i\to j,e}+x_e=a_e(1+y_{i\to j,e}x_e)\), that is,
\(y_{i\to j,e}=h_{a_e}(x_e)\) with \(h_a(x)=(a-x)/(1-ax)\). On the open
set where \(a_e\ne0\), \(a_e^2\ne1\), and \(x_e\ne0,a_e,a_e^{-1}\), the
map \(h_{a_e}\) is an involution with nonzero values, and
\(a_eh_{a_e}(x_e)-1=(a_e^2-1)/(1-a_ex_e)\ne0\), so both update
denominators on \(e\) are nonzero and the reverse update holds. Setting
\(b_i=-\bigl(\prod_{e\ni i}y_{\mathrm{other}\to i,e}\bigr)^{-1}\) gives
\(F_i=-1\) and all BP equations, and every solution with \(F\equiv-1\)
arises uniquely in this way. The total space of such solutions over all
parameters is therefore an open subset of the torus with coordinates
\((a_e,x_e)\), smooth and irreducible of dimension \(2m\), and the
parameter map is \((a,x)\mapsto(a,b(a,x))\).

If \(\nu(G)\ge1\), this map is dominant: its differential has rank
\(m+n\) at the point \(a_e=3\), \(x_e=1\) for all \(e\), where \(h_3(1)=-1\) and
\(b_i=\pm1\). Indeed, \(\partial\log h_a(x)/\partial\log x=
x(a^2-1)/((1-ax)(a-x))=-2\) there, so the matrix
\(\partial\log b/\partial\log x\) has, in the column of \(e=i\to j\), the
entry \(-1\) in row \(i\), the entry \(2\) in row \(j\), and zeros
elsewhere. Choose a spanning tree and one further edge, giving a
spanning subgraph with a single cycle of length \(\ell\ge2\); orient the
cycle cyclically and the tree edges away from it. In the corresponding
\(n\times n\) minor, each pendant vertex is the head of its unique edge
and its row has the single entry \(2\); expanding along these rows
removes a factor \(\pm2\) each and leaves the \(\ell\times\ell\) cycle
block \(2P-I\) for a cyclic permutation matrix \(P\), whose determinant
is \(\pm(2^\ell-1)\). The minor equals \(\pm2^{n-\ell}(2^\ell-1)\ne0\),
and together with the identity on the \(a\)-coordinates the rank is
\(m+n\).

For \(\nu(G)\ge1\), dominance implies that the image contains a nonempty
open set. Shrinking this open set, generic smoothness in characteristic
zero makes the map smooth of relative dimension \(2m-(m+n)=\nu(G)-1\).
For a tree, \(2m<m+n\), so the closure
of the image is a proper subset and the generic fiber is empty.
\end{proof}

\section{Supplementary proofs for the edge-incidence ring}
\label{supp:ring}

\subsection{The incidence scheme as a vector bundle}
\label{supp:vector-bundle}

\begin{proof}[Proof of \cref{alg:incidence-vector-bundle}]
The evaluation map
\[
 \mathbb B_G\otimes\cO_{X_G}\longrightarrow F_G
\]
is surjective, since every edge bundle \(L_e\) is globally generated.
Its kernel is locally free of rank \(4m-m=3m\), and its total space
is exactly \(Z_G\).  In particular, the universal edge equations form
a regular sequence on \(X_G\times\mathbb B_G\): locally, after
splitting the evaluation map, they are \(m\) independent linear
coordinates in the coefficient directions.

The scheme \(Z_G\) is closed in \(X_G\times\mathbb B_G\), and
\(X_G\) is projective.  Hence \(\pi_G\) is proper and
\((\pi_G)_*\cO_{Z_G}\) is coherent, proving module finiteness over
the affine coefficient space.  A function of coefficient degree zero
is a function on the zero section \(X_G\); these functions form \(\C\).
\end{proof}

\subsection{Cohen--Macaulayness at every prime}
\label{supp:cm-localization}

The proof of \cref{alg:normal-cm} established the equality of depth and
dimension at the homogeneous maximal ideal. The following direct
localization argument gives the Cohen--Macaulay property at every prime
without appealing to the graded reduction.
For completeness, this also gives Cohen--Macaulayness away from the
homogeneous maximal ideal.  For every prime \(\mathfrak p\supseteq I\),
localizing the same resolution yields
\[
 \depth_{(S_G)_{\mathfrak p}}(\mathscr U_G)_{\mathfrak p}
 \geq\operatorname{ht}\mathfrak p-\rho
 =\dim(\mathscr U_G)_{\mathfrak p}.
\]
The final equality is the dimension formula for the prime \(I\) in
a polynomial ring.  Thus \(\mathscr U_G\) is a Cohen--Macaulay
\(S_G\)-module, equivalently a Cohen--Macaulay finite
\(S_G\)-algebra.  The depth inequality used here follows successively
from the depth lemma \cite[Tag~00LX]{stacks2026cohomology}.

\subsection{Functoriality}
\label{supp:functoriality}

\begin{proof}[Proof of \cref{alg:functoriality}]
A coefficient on an image edge of \(H\) can be copied to every edge
of \(G\) above it, renaming the two endpoint coordinates and using
the map of orientation lines.  This defines a linear map
\(\beta_\phi:\mathbb B_H\to\mathbb B_G\).  In oriented coordinates,
the copied sections satisfy
\[
 a_\phi^*b_e=\epsilon_\phi(e)b_{\phi(e)}.
\]
The pair \((a_\phi,\beta_\phi)\) therefore maps \(Z_H\) into
\(Z_G\).  Pullback of regular functions gives \(\mathscr U_\phi\).
Both coordinate maps and coefficient-copy maps compose, proving the
functor laws.  For a disjoint union the incidence scheme is the product
of the two incidence schemes.  K\"unneth for global sections over
\(\C\) gives the tensor-product formula, compatibly with its
associativity, symmetry and unit.

There is also a direct Koszul description.  Pull back the edge
generators of \(G\), and send the generator for \(e\) to
\(\epsilon_\phi(e)\) times the generator for \(\phi(e)\).
The compatibility of sections makes this a map of Koszul algebras.
Derived global sections give
\[
 \mathscr U_G\otimes^{\mathbf L}_{S_G}S_H
 \longrightarrow\mathscr U_H.
\]
After derived tensoring with \(\C\) over \(S_H\), this is the
natural map between the two Tor algebras.  Specializing both coefficient
spaces to the origin gives exactly the map on
\eqref{alg:zero-base-change}; hence \eqref{alg:tor-isomorphism} is
natural and preserves products.  This argument does not require the
coefficient-ring map to be flat.
\end{proof}

\subsection{The cycle-compatibility algebra}
\label{supp:cycle-product}

\begin{proof}[Proof of \cref{alg:cycle-product}]
The left-hand side is \(\Tor_0^{S_G}(\mathscr U_G,\C)\).
A contributing state in homological degree zero satisfies
\(|S|=v(S)\).  Since every active vertex has degree at least two,
all these degrees must equal two.  Such a state is exactly a collection
of pairwise vertex-disjoint simple cycles, and its cohomology is
one-dimensional.  Choose a nonzero class \(z_C\) for each single
cycle.

If two cycles share an edge, their exterior product vanishes.  If they
share a vertex but no edge, their cup product at that \(\Pp^1\)
factor lies in \(H^2\) and again vanishes.  For vertex-disjoint
cycles the product is nonzero by K\"unneth and spans the state of
their union.  These products account for every state in degree zero,
and different collections give distinct state summands.  Thus the
displayed monomial relations generate the entire kernel.  All classes
in this proposition have total cohomological degree zero, so their
product is ordinarily commutative.
\end{proof}

\subsection{Module generators and real sectors}
\label{supp:generator-bridge}

\begin{proof}[Proof of \cref{alg:generator-real-bridge}]
By \cref{alg:cycle-product}, the generators are indexed by supports
whose vertex degrees are zero or two.  They are a subset of the binary
cycle space, consisting of all supports with even vertex degrees;
that space has \(2^{m-n+c}=2^\nu\) elements.  Equality therefore
means that every even support has maximum degree at most two.
The union of two edge-disjoint cycles meeting at a vertex violates
this condition.  Conversely, an even support of degree at least four
at some vertex decomposes into edge-disjoint simple cycles, at least
two of which meet there.  If the graph has maximum degree at most
three, every even support automatically has degrees zero or two.
\end{proof}

\subsection{Three invariants of one resolution}
\label{supp:three-invariants}

\begin{proof}[Proof of \cref{alg:hilbert}]
The recurrence for \(f_d\) gives
\(f_d(2\mathrm i)=\mathrm i^d(1-d)\) for all \(d\geq0\).
For a contributing state \(S\), therefore,
\[
 \prod_v f_{d_S(v)}(2\mathrm i)
 =(-1)^{|S|+v(S)}\wt(S)
 =(-1)^{|S|-v(S)}\wt(S).
\]
The same equality holds with both sides zero for a state with a
degree-one vertex.  Substituting \eqref{alg:betti-state-sum} proves
\eqref{alg:theta-numerator}.  The alternating Hilbert series of the
free resolution then proves the first equality in
\eqref{alg:universal-hilbert}.  The second follows from
\eqref{alg:omega-definition}.

By \cref{alg:normal-cm}, \(\mathscr U_G\) is Cohen--Macaulay of
dimension \(d_G\).  Its finiteness over \(S_G\) implies that
\(d_G\) general coefficient linear forms are a homogeneous system of
parameters.  Indeed, successive general hyperplanes cut the support in
coefficient space to the origin.  On a Cohen--Macaulay algebra a system
of parameters is regular \cite[Tag~00N7]{stacks2026cohomology}.
Each linear nonzerodivisor multiplies the Hilbert series of the quotient
by \(1-t\).  Applying this \(d_G\) times gives
\eqref{alg:artin-hilbert}.
\end{proof}

\subsection{One-dimensional sections}
\label{supp:pencil}

\begin{remark}[One-dimensional sections]
\label{alg:pencil-sections}
If \(d_G>0\), cutting first by \(d_G-1\) general coefficient linear forms gives
a one-dimensional Cohen--Macaulay algebra \(M_G\) over
\(R\cong\C[\lambda_1,\ldots,\lambda_{\rho+1}]\), with
\(\Hilb_{M_G}(t)=\romega_G(t)/(1-t)\).
The sequence is regular, so tensoring the resolution with \(R\)
preserves exactness and minimality.  Thus its Betti table is unchanged.
One further regular linear cut gives \(A_G\) over \(R/(\ell)\),
again with the same Betti table and total Betti number \(N(G)\).
The isomorphism classes of these sections can depend on their choices;
their Hilbert series and Betti tables do not.
\end{remark}

\subsection{Sharp bounds at fixed cycle rank}
\label{supp:degree-bounds}

\begin{proposition}[Sharp bounds at fixed cycle rank]
\label{alg:degree-bounds}
For a connected loopless multigraph with at least one vertex,
\[
 (\nu-1)2^\nu+2\ \leq\ N(G)=Q(G)\ \leq
 (4+2\sqrt2)^{\nu-1}+(4-2\sqrt2)^{\nu-1}.
\]
Both bounds are attained by simple graphs for every \(\nu\geq0\).
Equality in the upper bound holds exactly when the 2-core has maximum
degree at most three.
\end{proposition}

\begin{proof}
For the lower bound only, extend the combinatorial sum \(Q\) to
graphs with loops, counting a loop twice in the degree.  For a nonloop
edge \(uv\), pair the states omitting and containing this edge.
Writing \(a,b\) for the other degrees at its endpoints, the identity
\[
 \varphi(a)\varphi(b)+ab-\varphi(a+b)
 =2\mathbf1_{\{a,b\geq2\}}\varphi(a)\varphi(b)\geq0
\]
gives \(Q(G)\geq Q(G/uv)\).  Contract a spanning tree.  The resulting
bouquet of \(\nu\) loops has sum
\(\sum_{k=0}^\nu\binom{\nu}{k}\varphi(2k)=(\nu-1)2^\nu+2\).

For the upper bound, \(\varphi(d)\leq f_d(2)\), with equality
exactly for \(0\leq d\leq3\).  Thus \(Q(G)\leq\theta_G(1,2)\).
For a connected graph the value \(\theta_G(1,\gamma)\) depends only
on \(\nu\).  Indeed, with \(\theta\) extended to graphs with loops as
in \cite{watanabe2011graphpolynomials}, a loop again counting twice
in the degree, the deletion--contraction formula
\(\theta_G(t,\gamma)=(1-t)\theta_{G-e}(t,\gamma)+t\theta_{G/e}(t,\gamma)\)
for a nonloop edge \(e\)
\cite[Theorem~2.3]{watanabe2011graphpolynomials} reduces at
\(t=1\) to \(\theta_G(1,\gamma)=\theta_{G/e}(1,\gamma)\).
Contracting a spanning tree leaves the bouquet of \(\nu\) loops, so
\begin{equation}
 \theta_G(1,\gamma)
 =\sum_{k=0}^{\nu}\binom{\nu}{k}f_{2k}(\gamma)
 =(\xi+\xi^{-1})^{\nu-1}\bigl(\xi^{\nu-1}+\xi^{1-\nu}\bigr),
 \qquad \gamma=\xi-\xi^{-1},
 \label{alg:theta-special-value}
\end{equation}
which is \cite[Lemma~4.2, equation~(4.1)]{watanabe2011graphpolynomials}.
The second equality follows from the closed form
\(f_d(\gamma)=(\xi^{d}+(-1)^d\xi^{2-d})/(1+\xi^2)\) of the recurrence
and the binomial theorem.  For \(\gamma=2\) take \(\xi=1+\sqrt2\);
then \((\xi+\xi^{-1})\xi^{\pm1}=4\pm2\sqrt2\), which gives the
displayed upper bound.
Both sums ignore pendant edges.
On the 2-core, every term agrees if its maximum degree is at most
three; otherwise the full-core state gives a strict inequality.
The lower bound is attained by \(\nu\) triangles sharing one vertex,
and the upper by a ladder with \(\nu\) squares; for \(\nu=0\),
any tree attains both bounds.
\end{proof}

\subsection{The objects being compared}
\label{supp:scope}

\begin{remark}[The objects being compared]
\label{alg:scope}
The universal edge-incidence ring records common zeros of independently
varying edge bilinear forms.  Its Artin length \(\romega_G(1)\) and its total
Betti number \(Q(G)\) are different invariants, as the values 30 and
48 above illustrate.  The equality of the latter with \(N(G)\)
comes from the BP counting theorem, not from an identification of the
Artin section with a BP root algebra.  Likewise, the naturality proved
here concerns noncontracting graph maps; a general edge contraction is
outside this category.  In the next section, the same state factors
will be identified with local multiplicities of the actual BP family
at a boundary.  That comparison concerns the state factors and their
local bases, and does not identify its deformation differential with
the differential of an edge Koszul complex.
\end{remark}

\subsection{The Gr\"obner-basis check for \(K_4\)}
\label{supp:k4-groebner}

The computation cited in \cref{alg:k4-section} is the following
Singular (version 4.4.1) script, run on 2026-10-12.  The unknowns
\(y_0,\dots,y_{11}\) are the twelve directed messages of \(K_4\),
\(X_k\) are the cavity products, the parameters \(a_e\) and
\(b_i\) are the fixed residues modulo the prime \(32003\) listed in
the script, and the ideal is saturated with respect to every
message and every update denominator before the standard basis is
computed.
{\small\begin{verbatim}
LIB "elim.lib";
ring r = 32003,(y0,y1,y2,y3,y4,y5,y6,y7,y8,y9,y10,y11),dp;
poly X0 = 8360*y3*y5;
poly X1 = 3865*y7*y9;
poly X2 = 8360*y1*y5;
poly X3 = 16236*y6*y11;
poly X4 = 8360*y1*y3;
poly X5 = 24937*y8*y10;
poly X6 = 3865*y0*y9;
poly X7 = 16236*y2*y11;
poly X8 = 3865*y0*y7;
poly X9 = 24937*y4*y10;
poly X10 = 16236*y2*y6;
poly X11 = 24937*y4*y8;
ideal I = y0*(X0+4404)-(4404*X0+1), y1*(X1+4404)-(4404*X1+1),
  y2*(X2+18653)-(18653*X2+1), y3*(X3+18653)-(18653*X3+1),
  y4*(X4+27770)-(27770*X4+1), y5*(X5+27770)-(27770*X5+1),
  y6*(X6+26290)-(26290*X6+1), y7*(X7+26290)-(26290*X7+1),
  y8*(X8+25029)-(25029*X8+1), y9*(X9+25029)-(25029*X9+1),
  y10*(X10+2069)-(2069*X10+1), y11*(X11+2069)-(2069*X11+1);
list L = y0, X0+4404, y1, X1+4404, y2, X2+18653, y3, X3+18653,
  y4, X4+27770, y5, X5+27770, y6, X6+26290, y7, X7+26290, y8,
  X8+25029, y9, X9+25029, y10, X10+2069, y11, X11+2069,
  1+8360*y1*y3*y5, 1+3865*y0*y7*y9, 1+16236*y2*y6*y11,
  1+24937*y4*y8*y10;
int k; for (k=1; k<=size(L); k++) { def S = sat(I, L[k]);
  if (typeof(S)=="list") { I = S[1]; } else { I = S; } kill S; }
I = std(I);
"graph K4 mode modp" ;
"p=32003 a=4404,18653,27770,26290,25029,2069 b=8360,3865,16236,24937";
"dim:"; dim(I);
"vdim:"; vdim(I);
ideal E = y0*(X0+4404)-(4404*X0+1), y1*(X1+4404)-(4404*X1+1),
  y2*(X2+18653)-(18653*X2+1), y3*(X3+18653)-(18653*X3+1),
  y4*(X4+27770)-(27770*X4+1), y5*(X5+27770)-(27770*X5+1),
  y6*(X6+26290)-(26290*X6+1), y7*(X7+26290)-(26290*X7+1),
  y8*(X8+25029)-(25029*X8+1), y9*(X9+25029)-(25029*X9+1),
  y10*(X10+2069)-(2069*X10+1), y11*(X11+2069)-(2069*X11+1);
ideal J = I, det(jacob(E));
"solutions with singular Jacobian (vdim of I + det Jac;
  0 means all simple):"; vdim(std(J));
quit;
\end{verbatim}}
The saved output is
{\small\begin{verbatim}
graph K4 mode modp
p=32003 a=4404,18653,27770,26290,25029,2069 b=8360,3865,16236,24937
dim:
0
vdim:
48
solutions with singular Jacobian (vdim of I + det Jac; 0 means all simple):
\end{verbatim}}
so the saturated ideal is zero-dimensional of length \(48\) over
\(\F_{32003}\).  The last query, the length of the ideal with the
Jacobian determinant adjoined, was interrupted and produced no output;
it is not used.  This is a finite-characteristic specialization check
with multiplicity, and the proof of \(N=Q\) does not depend on it.

\section{Supplementary proofs for the strong-field boundary}
\label{supp:strong-field}

\subsection{The algebra at a state}
\label{supp:boundary-algebra}

\begin{proof}[Proof of \cref{sf:boundary-algebra}]
The chart is defined over the parameter space by the regular sequence
$(-1)^{k_i}w_i^{r_i}-p_i$.
Its monomials $\prod_i w_i^{a_i}$ with $0\le a_i<r_i$ form a free
basis over the convergent parameter ring $\C\{p_i:i\in V\}$.
Setting $p=0$ gives \eqref{sf:local-algebra} and the asserted length.
The $i$th power equation gives the claimed ramification index.
\end{proof}

\subsection{Duality and coordinate choices}
\label{supp:duality}

The statewise comparison of \cref{sf:boundary-algebra} with the
cohomology table can be phrased through Serre duality; the following
remark records what this does and does not provide.

\begin{remark}[The role of duality and coordinate choices]
\label{sf:serre-comparison}
For an active vertex, $k\ge2$, Serre duality
\cite[Tag 0BS4]{stacks2026cohomology} gives
\[
 H^1(\Pp^1,\cO(-k))
 \cong H^0(\Pp^1,\cO(k)\otimes\omega_{\Pp^1})^*.
\]
After identifying $\omega_{\Pp^1}$ with $\cO(-2)$ and choosing an
affine coordinate and a trivialization, the space on the right is the
dual of polynomials of degree at most $r-1$, where $r=k-1$.
This is the dual of $B=\C[w]/(w^r)$ as a vector space. The perfect
pairing $(f,g)\mapsto[w^{r-1}](fg)$ then identifies $B$ with $B^*$
and satisfies $(w^aB)^\perp=w^{r-a}B$ for $0\le a\le r$.
It expresses the comparison through dual filtrations, but the scalar
pairing uses the chosen coordinate. Duality alone therefore does not
give a canonical algebra identification with the Tor summand.
\end{remark}

\subsection{The eight real states of \(K_4\)}
\label{supp:k4-states}

\begin{example}[The eight real states of $K_4$]
\label{sf:k4-states}
The state types in \cref{fig:k4-states} give $1+4+3+6\cdot4+16=48$
complex branches. The Eulerian states are $\varnothing$, the four
triangles, and the three four-cycles; every active vertex in them has
degree two, so every $r_i=1$ and each boundary algebra has length one.
For positive $a_e$ and small positive $p_i$ each gives one real branch,
positive only for $S=\varnothing$ and otherwise with negative messages
in both directions on the selected edges. The six $K_4-e$ states have
two equations $w_i^2=-p_i$ and the full state four, so they have no real
branches in this chamber. Thus $(N_{\C},R_{K_4},N_+)=(48,8,1)$, and
$8=1+4+3=\sum_q\beta_{0,q}^{S_{K_4}}(\U_{K_4})$ counts all real sectors.
\Cref{real:strong-field-count} gives the same count in every field
orthant; at general fields each sector remains nonempty, while
uniqueness within a sector is asserted only at strong field.
\end{example}

\section{Supplementary proofs for the real sectors}
\label{supp:real}

\subsection{The two real edge branches}
\label{supp:edge-branches}

\begin{proof}[Proof of \cref{real:edge-branches}]
Write $a=a_e$, $s=a^2$, and $A_i=X_{i,e}/a$, $A_j=X_{j,e}/a$.
The edge equations are
\begin{equation}
 F_i=A_i\frac{1+sA_j}{1+A_j},\qquad
 F_j=A_j\frac{1+sA_i}{1+A_i}.
 \label{real:edge-cavities}
\end{equation}
Eliminating $A_j$ gives
\begin{equation}
 sA_i^2+C_iA_i-F_i=0,\quad
 C_i=\frac{1+sF_j-F_i(s+F_j)}{1+F_j},\quad
 A_j=F_j\frac{1+A_i}{1+sA_i}.
 \label{real:edge-quadratic}
\end{equation}
The discriminant is $C_i^2+4sF_i>0$, and the product of the roots
is $-F_i/s<0$.  Hence one root is positive and the other negative.
For the polynomial $P_i$ on the left,
\[
 P_i(-1)=\frac{(s-1)(1+F_i)}{1+F_j},\qquad
 P_i(-1/s)=\frac{(1-s)F_j(1+F_i)}{s(1+F_j)}.
\]
These values have opposite signs.  The negative root is strictly
between $-1$ and $-1/s$, so neither excluded denominator vanishes,
and the corresponding $A_j$ is negative.
The positive root gives $A_j>0$.
Conversely, any admissible solution of \eqref{real:edge-cavities}
must appear in \eqref{real:edge-quadratic}.  Neither reconstructed
$A_j$ is $-1$ or $-1/s$: substituting either value into the cleared
identity $F_i(1+A_j)=A_i(1+sA_j)$ contradicts $s\ne1$,
$A_i\ne0$, and $F_i>0$.
Thus both reconstructed pairs solve the original rational
equations.

The discriminant and every elimination denominator are nonzero,
so the two branches are real analytic.  Solving the second
equation for $A_j$ and then the first by the simple quadratic
shows that the edge Jacobian is invertible.
Finally, $y_{j\to i,e}^{\pm}=F_i/(aA_i^{\pm})$; the product of the two
quadratic roots gives \eqref{real:edge-product}.
\end{proof}

\subsection{Vertex coordinates on a sector}
\label{supp:vertex-map}

\begin{proof}[Proof of \cref{real:vertex-map}]
Given $u$, prescribe $F_i=e^{2u_i}$ and select the positive edge
branch off $S$ and the negative branch on $S$.
The incoming signs multiply to one at every vertex.  By
\eqref{real:log-branches},
\[
 \sum_{e\ni i}\tfrac12\log|y_{j\to i,e}|
 =k_i u_i+\sum_{e\ni i}\varepsilon_e z_{e,i}^{+}(u).
\]
Hence the consistency equation $F_i=b_i\prod_{j\to i,e}y_{j\to i,e}$
is equivalent to \eqref{real:vertex-equation}.
When it holds, the cavity ratios reconstructed from the edge
branches agree with the original cavity products, and all BP
equations hold.  The edge denominators and messages are nonzero
by Lemma~\ref{real:edge-branches}, and $1+F_i>1$, so these roots
are admissible.

Conversely, Proposition~\ref{real:sign-sectors} assigns to any
admissible real root positive $F_i$ and an Eulerian support $S$.
On every edge Lemma~\ref{real:edge-branches} forces the selected
branch; the root therefore recovers \eqref{real:vertex-equation}.
For the local assertion, adjoin the consistency equations for
$F_i$ to the BP equations and eliminate the edge messages using
the invertible edge Jacobians.  This leaves
\eqref{real:vertex-equation} after invertible analytic coordinate
changes.  The local solution spaces and their Jacobian
nondegeneracy are consequently preserved.
Finally, \eqref{real:global-box} follows from
\eqref{real:field-map} and \eqref{real:edge-bound}.
\end{proof}

\subsection{A degree budget for positive solutions}
\label{supp:positive-bound}

\begin{proof}[Proof of \cref{real:positive-bound}]
There are $Q(G)$ distinct complex roots by
Theorem~\ref{thm:degree}.  Each of the $2^{\nu(G)}-1$ nonempty
Eulerian sectors contains at least one real root, all of them
nonpositive, by Proposition~\ref{real:sector-degree}.
Subtracting these roots from the complex total gives the upper
bound for $N_+$.  The empty sector supplies at least one positive
root, and the remaining inequalities follow at once.
Complex conjugation pairs the nonreal roots of this real, reduced
fiber, proving the first congruence in \eqref{real:parity}.  For the
second, $\wt(S)=\prod_v\varphi(d_S(v))$ is odd exactly when every
$d_S(v)$ is even, that is, when $S$ is Eulerian, and there are
$2^{\nu(G)}$ Eulerian subsets.
\end{proof}

\subsection{Zero couplings}
\label{supp:zero-couplings}

The nonzero-coupling hypothesis is essential for the lower bound
with the cycle rank of $G$.  If $J_e=0$, the update on that edge
is $y_{i\to j,e}=y_{j\to i,e}=1$ on the admissible real set.
Deleting all such edges gives a bijection of admissible real
solution sets with the resulting support graph: in the reverse
construction the new denominators are $F_i+1>0$ by
Proposition~\ref{real:sign-sectors}.
The lower bound and the minimum formula then use the cycle rank
of that support graph.

\section{Further comparison with related work}
\label{supp:related}

This section expands the comparisons of \cref{intro:related}.

\paragraph{Variational characterizations}
Yedidia, Freeman, and Weiss identify BP fixed points with interior
stationary points of the constrained Bethe free energy
\cite{yedidia2005bethe}, and Watanabe classifies the signed graphs on
which the positive fixed point is unique for all interaction magnitudes
and fields, using an index sum over nondegenerate stationary points
\cite{watanabe2011signed}. Both concern positive messages, whereas the
sign sectors of \cref{sec:real} live in the rational BP equations, where
negative messages are admissible and are counted by the Eulerian edge
sets.

\paragraph{Enumeration by polynomial systems}
Knoll, Mehta, Chen, and Pernkopf encode normalized messages and
normalization factors by polynomial equations and use numerical
polynomial homotopy continuation to enumerate isolated solutions
\cite{knoll2018bp}. Their construction uses the sparsity of these
equations through a Bernstein--Kushnirenko--Khovanskii bound, and
positivity selects the probabilistic fixed points. Here the enumerative
question is the generic degree of the admissible ratio equations within
a prescribed pairwise parameter family, with independent edge and field
parameters and a fixed factorization. Genericity in this family is
distinct from genericity among all coefficients with a prescribed
monomial support, and the two counts are not compared numerically in
this paper. Three differences are structural. Messages normalized to
sum one and the ratio coordinates used here are related by the
invertible change \(y\mapsto(y/(1+y),1/(1+y))\) on the common open set
where no message equals \(-1\). A system written in message variables
alone, without saturating the vertex normalizations \(1+F_i\), retains
the non-normalizable locus \(F_i\equiv-1\), which for a connected graph
with \(\nu(G)\ge2\) is positive-dimensional for generic parameters
(\cref{supp:nonnormalizable-dim}) and already contributes two of the
four solutions in \cref{ex:triangle-admissible}. Finally, a
Bernstein--Kushnirenko--Khovanskii bound counts isolated solutions of a
system with generic coefficients of the given support, whereas \(N(G)\)
is the generic degree within the family of Ising-type parameters with
the admissibility conditions imposed. Consequently a difference between
the two numbers is not a count of extra isolated solutions, and this
paper does not assert any statement about the computations of
\cite{knoll2018bp}.

\paragraph{Loop calculus and graph polynomials}
Chertkov and Chernyak express the partition function as a finite
correction series around a BP solution, with contributions supported on
subgraphs having no vertices of degree one
\cite{chertkov2006loopcalculusv2}. These weights depend on the chosen
BP solution and serve a different enumerative purpose from a count of
solutions. Watanabe and Fukumizu established the polynomiality of
\(\omega_G\), its combinatorial expansions, and the nonnegativity of
its coefficients for graphs without tree components
\cite{watanabe2011graphpolynomials}. The total Betti number of the
resolution in \cref{sec:incidence} equals the generic BP degree, whereas
the length \(\romega_G(1)\) of the Artin reduction is a separate
invariant.

\paragraph{Geometric syzygies and resultants}
Bender, Faug\`ere, Mantzaflaris, and Tsigaridas give explicit Weyman
complexes and determinantal formulas for mixed multilinear systems
\cite{bender2021mixedmultilinearv1}. Their resultant construction treats
a multihomogeneous system with one more equation than the dimension of
the base, and their bipartite bilinear systems provide a further
relevant class of supports. For the edge-incidence family the explicit
minimal resolution, its multiplicative structure, and the normal
Cohen--Macaulay ring are established directly, without a resultant
hypothesis on the number of equations. The universal edge-incidence ring
and the coordinate algebra of a BP fiber have different roles in this
construction.

\paragraph{Homological and topological approaches to message passing}
Peltre builds a complex of local observables and formulates belief
propagation as a transport equation with homological conservation laws
\cite{peltre2019homologicalv1}. Sergeant-Perthuis and Boitel show how
deformation retractions of factor-graph posets, together with explicitly
transformed Hamiltonians, preserve Bethe critical points, with
corresponding results for minima
\cite{sergeantperthuis2025retractionsv1}. These constructions describe
message dynamics and positive variational problems. The complex used
here is the sheaf-theoretic Koszul complex on a product of projective
lines, whose cohomology controls a resolution over the universal
coefficient ring; the connection to the actual BP family is provided by
the common state weights \(\wt(S)\) and by the local power-map
description at strong field. No identification of these different
complexes, and no general invariance of \(N(G)\) under graph homotopy,
is used.